\documentclass[12pt]{article}

\usepackage{arxiv}

\usepackage{amsmath}
\usepackage{amssymb}
\usepackage{amsthm}
\usepackage{amsfonts}
\usepackage{multirow}
\usepackage{xcolor}
\usepackage{graphicx}
\usepackage{natbib}
\usepackage{url} 

\newtheorem{theorem}{Theorem}
\newtheorem{lemma}{Lemma}

\newcommand{\blind}{0}

\begin{document}




\if0\blind
{
  \title{\bf Poisson PCA for matrix count data}
  \author{Joni Virta\thanks{
    The work of JV was supported by the Academy of Finland (Grant 335077). The work was supported by LMS Research in Pairs (Grant \#41821) awarded to AA to initiate the collaboration with JV. }\hspace{.2cm}\\
    Department of Mathematics and Statistics, University of Turku\\
    and \\
    \textbf{Andreas Artemiou} \\
    School of Mathematics, Cardiff University}
  \maketitle
} \fi

\if1\blind
{
  \bigskip
  \bigskip
  \bigskip
  \begin{center}
    {\LARGE\bf Poisson PCA for matrix count data}
\end{center}
  \medskip
} \fi

\bigskip
\begin{abstract}
We develop a dimension reduction framework for data consisting of matrices of counts. Our model is based on assuming the existence of a small amount of independent normal latent variables that drive the dependency structure of the observed data, and can be seen as the exact discrete analogue for a contaminated low-rank matrix normal model. We derive estimators for the model parameters and establish their root-$n$ consistency. An extension of a recent proposal from the literature is used to estimate the latent dimension of the model. Additionally, a sparsity-accommodating variant of the model is considered. The method is shown to surpass both its vectorization-based competitors and matrix methods assuming the continuity of the data distribution in analysing simulated data and real abundance data.
\end{abstract}

\noindent%
{\it Keywords:} Poisson log-normal distribution, Kronecker model, matrix normal distribution, zero-inflated Poisson distribution


\section{Introduction}

Modern applications typically see data with structures going significantly beyond the traditional ``$n$ observations of $p$ continuous variables'' framework. In this work, our focus is on data where the observations $X_i = (x_{i, jk})$, $i = 1, \ldots , n$, $j = 1, \ldots , p_1$, $k = 1, \ldots , p_2$, are $p_1 \times p_2$ matrices of non-negative counts. Such data appear naturally, for example, in the analysis of publication data ($x_{i, jk}$ describes the word count of the $j$th word for the $i$th author at the $k$th venue) \citep{hu2015scalable}, abundance studies ($x_{i,jk}$ describes the abundance of the $i$th species in the $j$th location at the $k$th time peroid) \citep{frelat2017community} and in the analysis of dyadic events ($x_{i,jk}$ describes the number of actions initiated by the $j$th actor targeting the $k$th actor during the $i$th time period) \citep{schein2015bayesian}. A common thread to all these applications is that the involved data sets are usually both large in size and inherently complex. As such, a natural first step in their analysis is dimension reduction, which both helps reduce their size and allows interpreting the data through the discovered latent variables. The development of a natural framework for the dimension reduction of matrix-valued count data is thus the objective of the current work.

In order for us to succeed in this task, the developed methods have to naturally accommodate the two main features of the data: the matrix structure and the discreteness of the observations. Ignoring the latter of these for a second, the recent decade has seen an ever-increasing amount of standard multivariate statistical methods being generalized to allow for matrix-valued data. A majority of these extensions uses the so-called \textit{Kronecker approach} to modelling which we next exemplify in the simple context of a linear latent factor model.

Given a $p_1 \times p_2$ random matrix $X$, the ``naive'' approach to latent factor modeling would be to vectorize $X$ to obtain the $(p_1 p_2)$-dimensional vector $\mathrm{vec}(X)$ ($\mathrm{vec}$ stacks the columns of its input to a long vector) and assume, for example, that
\begin{align}\label{eq:linear_latent_vector_model}
    \mathrm{vec}(X) = \mu_0 + U_0 z + \varepsilon_0,
\end{align}
where $\mu_0 \in \mathbb{R}^{p_1 p_2}$, $U_0 \in \mathbb{R}^{p_1 p_2 \times d}$ are unknown parameters and the $d$-dimensional random vector $z$ and the $(p_1 p_2)$-dimensional random vector $\varepsilon_0$ signify the latent signal and the noise, respectively. Whereas, under the Kronecker approach, we would instead preserve the matrix structure of $X$ and assume that,
\begin{align}\label{eq:linear_latent_matrix_model}
    X = \mu + U_1 Z U_2' + \varepsilon,
\end{align}
where $\mu \in \mathbb{R}^{p_1 \times p_2}$, $U_1 \in \mathbb{R}^{p_1 \times d_1}$, $U_2 \in \mathbb{R}^{p_2 \times d_2}$ are unknown parameters and the random $d_1 \times d_2$ matrix $Z$ and the random $p_1 \times p_2$ matrix $\varepsilon$ represent the latent signal and the noise, respectively. To understand the relationship between the two approaches, we apply the vectorization $\mathrm{vec}$ to the model \eqref{eq:linear_latent_matrix_model} and use the formula $\mathrm{vec}(AYB') = (B \otimes A) \mathrm{vec}(Y)$ to reveal that $d = d_1 d_2$ and that the $U$-parameters have the connection $U_0 = (U_2 \otimes U_1)$. Thus, we see that the matrix model \eqref{eq:linear_latent_matrix_model} is actually a \textit{constrained} version of the vectorized model \eqref{eq:linear_latent_vector_model} where the loading matrix $U_0$ is assumed to be a Kronecker product of two smaller matrices, $U_1$ which determines the dependency structure of the rows of $X$ and $U_2$ which governs the dependencies within the columns of $X$. Besides facilitating the previous natural intepretation, this constraining also works as a form of \textit{regularization} that helps avoid overfitting (see Section \ref{sec:examples} for an example). Indeed, the loading matrices $U_1$, $U_2$ in the model \eqref{eq:linear_latent_matrix_model} have a total of $p_1 d_1 + p_2 d_2$ parameters whereas the unconstrained $U_0$ in the vector model has $p_1 p_2 d_1 d_2$ parameters, a significantly larger amount already for moderately large dimensions. 
 
The Kronecker approach has been used to great success in developing matrix versions of, e.g., principal component analysis (PCA) \citep{zhang20052d,hung2012multilinear,ding2014dimension}, independent component analysis \citep{virta2017independent} and sufficient dimension reduction \citep{li2010dimension,ding2015tensor}. Consequently, the Kronecker approach will also be our tool of choice. Note that the same idea can also be seen to underlie several tensor decompositions, such as the higher order singular value decomposition (HOSVD) and the Tucker decomposition, see, e.g., \cite{de2000multilinear,kolda2009tensor}.

Moving on to the count aspect of our data, the error $\varepsilon$ in \eqref{eq:linear_latent_matrix_model} is typically Gaussian, implying that the model is not suitable for count data. Moreover, even if the error variable was taken to be discrete, some heavy constraints would need to be placed to the model parameters to guarantee that $X$ contains only non-negative integers, making the approach rather unnatural. As such, we will instead derive a discrete analogy for \eqref{eq:linear_latent_matrix_model}. Such discrete extensions of latent variable models have been intensively studied in the case of vector-valued data and the most popular approach is perhaps that of \textit{exponential family PCA}, where a latent variable model is assumed for the canonical parameter of an exponential family distribution, see, \cite{collins2001generalization, li2010simple, landgraf2015generalized, smallman2018sparse, smallman2019simple}. Another major framework for the modelling of count data, and the one on which we base our extension of \eqref{eq:linear_latent_matrix_model}, is the Poisson log-normal (PLN) distribution where the data are taken to be conditionally Poisson distributed with the mean parameters following the log-normal distribution. Introduced originally in \cite{aitchison1989multivariate} (but without the dimension reduction context), the PLN model has since been studied from different viewpoints, see \cite{izsak2008maximum, hall2011theory, chiquet2018variational}, in particular from the perspective of variational inference. 

In this work, we extend the PLN model to matrix-valued data using the Kronecker approach, the obtained model retaining both the interpretability and the parsimonity of the Gaussian Kronecker model \eqref{eq:linear_latent_matrix_model}. We develop natural closed-form estimators for the model parameters using the method of moments, studying also their asymptotic behavior. As far as we are aware, large-sample results for PLN models have been proposed earlier only in \cite{hall2011theory} (in the vectorial case) and even then in a restrictive context requiring repeated measurements for each observational unit. We also use the parameter estimates to derive an estimator for the latent model dimension, based on the recently proposed idea of predictor augmentation \cite{luo2020order}. Additionally, we also develop a ``zero-inflated'' version of the method targeted for sparse count data, encountered commonly in practice, and extending our asymptotic results to this case also. Note that Bayesian models targeting the same type of data (counts of matrices) have been earlier studied in the literature \citep{hu2015scalable, liu2019time} but ours appears to be the first frequentist approach.


The manuscript is organized as follows. In Section \ref{sec:notation} we discuss briefly some notation and recall the matrix normal distribution that plays a key role in defining our model. Section~\ref{sec:tppca} discusses our model along with the estimation of its parameters, the latent variables and the latent dimensions. In Section~\ref{sec:examples} we demonstrate the method using both simulated data and an application to matrix-valued abundance data. The sparse data variant of the method is proposed in Section \ref{sec:zero_inflated} and we finally close with some discussion in Section \ref{sec:discussion}.

\section{Notation and some preliminaries}\label{sec:notation}

Throughout the manuscript, the subscripts 1 and 2 are used to refer to the left-hand and the right-hand sides of the model, respectively, as in \eqref{eq:linear_latent_matrix_model}. The $p_1 \times p_2$ matrix-variate normal distribution with the mean $\mu \in \mathbb{R}^{p_1 \times p_2}$ and the invertible left and right covariance matrices $\Sigma_1 \in \mathbb{R}^{p_1 \times p_1}$, $\Sigma_2 \in \mathbb{R}^{p_2 \times p_2}$ is denoted by $\mathcal{N}_{p_1 \times p_2}(\mu, \Sigma_1, \Sigma_2)$. That is, if $X \sim \mathcal{N}_{p_1 \times p_2}(\mu, \Sigma_1, \Sigma_2)$ then $X$ has the density function $f_X: \mathbb{R}^{p_1 \times p_2} \to \mathbb{R}$ defined as,
\begin{align*}
    f_X(X) = \frac{1}{(2 \pi)^{-p_1 p_2/2} |\Sigma_1|^{p_2/2} |\Sigma_2|^{p_1/2}} \exp \left[ -\frac{1}{2} \mathrm{tr} \left\{ \Sigma_1^{-1} (X - \mu) \Sigma_2^{-1} (X - \mu)' \right\} \right],
\end{align*}
see, for example, \cite{gupta2018matrix}. In a notable special case where the covariance matrices $\Sigma_1$ and $\Sigma_2$ are diagonal matrices, the elements of $X$ are mutually independent and the variance of the $(j, k)$th element equals the product $\sigma_{1, jj} \sigma_{2, kk}$ of the corresponding diagonal elements. Note also that the parameters $\Sigma_1$, $\Sigma_2$ are defined only up to the scaling $(\Sigma_1, \Sigma_2) \mapsto (t \Sigma_1, t^{-1} \Sigma_2)$ for any $t > 0$ (we get rid of this non-identifiability in the next section via a suitable reparametrization). 


Some of our results are more naturally formulated in terms of the (column) vectorizations of the related matrices. We use the convention that the vectorization of the observed matrix $X$ is denoted using the lower case $x := \mathrm{vec}(X)$ (and similarly for the latent matrix $Z$). A property we will repeatedly use without explicit mention is $\mathrm{vec}(AYB') = (B \otimes A) \mathrm{vec}(Y)$, valid for all matrices $A, B, Y$ with appropriate dimensions for $AYB'$ to be a valid matrix.

\section{Matrix Poisson PCA}\label{sec:tppca}

\subsection{Low-rank normal and Poisson models}\label{subsec:low_rank_models}

To motivate our proposed model for the dimension reduction of matrix count data, we first briefly review the analogous low-rank model for Gaussian data. Namely, assume that the observed $p_1 \times p_2$ random matrix $X$ is generated as
\begin{align}\label{eq:matrix_normal_model}
    X = \mu + U_1 Z U_2' + \varepsilon,
\end{align}
where $\mu \in \mathbb{R}^{p_1 \times p_2}$, $Z \sim \mathcal{N}_{d_1 \times d_2}(0, \tau \Lambda_1, \tau \Lambda_2)$, $\tau > 0$, $\Lambda_1 \in \mathbb{R}^{d_1 \times d_1}$ and $\Lambda_2 \in \mathbb{R}^{d_2 \times d_2}$ are positive-definite diagonal matrices satisfying $\mathrm{tr}(\Lambda_1) = p_1$ and $\mathrm{tr}(\Lambda_2) = p_2$ for some $d_1 < p_1$, $d_2 < p_2$, the error satisfies $ \varepsilon \sim \mathcal{N}_{p_1 \times p_2}(0, \sigma I_{p_1}, \sigma I_{p_2}) $ for some $\sigma > 0$ and $I_p$ denotes the $p \times p$ identity matrix. Furthermore, the random matrices $Z$ and $\varepsilon$ are assumed to be mutually independent. Alternatively, the same model can be written by requiring that $Z \sim \mathcal{N}_{d_1 \times d_2}(0, \tau \Lambda_1, \tau \Lambda_2)$ and
\begin{align}\label{eq:matrix_normal_model_conditional}
    X \mid Z \sim \mathcal{N}_{p_1 \times p_2}(\mu + U_1 Z U_2', \sigma I_{p_1}, \sigma I_{p_2})
\end{align}
In practice, the objective underlying this model is, given a sample from the distribution of $X$, to estimate the ``loading matrices'' $U_1$ and $U_2$ along with the corresponding latent matrices $Z$, achieving dimension reduction in the process. Naturally, the latent dimensions $d_1, d_2$ are usually unknown in practice and have to be estimated as well. As is common in the dimension reduction literature we separate this problem from the estimation of the other parameters and the latent components and assume, for now, that $d_1, d_2$ are known. Their estimation is then tackled later in Section \ref{subsec:order_determination}.

The mean matrix $\mu$ in \eqref{eq:matrix_normal_model} can be estimated through $\mathrm{E}(X)$ and the covariance parameters are standardly estimated using the higher-order singular value decomposition (HOSVD) \citep{de2000multilinear}, also known as 2D$^2$PCA \citep{zhang20052d}, where the matrix $U_1$ is found through the eigenvectors of the \textit{left covariance matrix} $\mathrm{Cov}_1(X) := (1/p_2)\mathrm{E}[ \{ X - \mathrm{E}(X) \} \{ X - E(X) \}' ]$. Namely, under model \eqref{eq:matrix_normal_model}, we have 
\begin{align*}
    \mathrm{Cov}_1(X) = \tau^2 U_1 \Lambda_1 U_1'  + \sigma^2 I_{p_1},
\end{align*}
showing that the leading eigenvectors of $\mathrm{Cov}_1(X)$ serve as an estimator for $U_1$ (or, in case of non-simple eigenvalues, for the corresponding subspace). The right-hand side matrix $U_2$ can be determined similarly by first transposing the observations, after which an estimate for the latent variables is obtained as $U_1' (X - \mu) U_2 = Z + \varepsilon_0$ where $\varepsilon_0 \sim \mathcal{N}_{d_1 \times d_2}(0, \sigma I_{d_1}, \sigma I_{d_2})$. Note that a noisy estimate is indeed the best we can do since the original observations are themselves contaminated with $\varepsilon$.

We base our matrix count model on the above ideas, similarly assuming the existence of a matrix of mutually independent normal latent variables $Z \sim \mathcal{N}_{d_1 \times d_2}(0, \tau \Lambda_1, \tau \Lambda_2)$ where the covariance parameters satisfy the trace constraints $\mathrm{tr}(\Lambda_1) = p_1$ and $\mathrm{tr}(\Lambda_2) = p_2$. Conditional on $Z$ the observed $p_1 \times p_2$ matrix $X$ of counts is assumed to satisfy
\begin{align}\label{eq:matrix_poisson_model_conditional}
    X \mid Z \sim \mathrm{Po}_{p_1 \times p_2}\{ \exp ( \mu + U_1 Z U_2' ) \},
\end{align}
where the parameters $\mu, U_1, U_2$ are as in \eqref{eq:matrix_normal_model_conditional}, the exponential function is applied element-wise and the notation $\mathrm{Po}_{p_1 \times p_2}(M)$, $M \in \mathbb{R}^{p_1 \times p_2}$, refers to a distribution on $p_1 \times p_2$ matrices having independent elements and whose $(j, k)$th component is Poisson-distributed with the mean $m_{jk}$. Comparison to \eqref{eq:matrix_normal_model_conditional} now reveals that the proposed model \eqref{eq:matrix_poisson_model_conditional} is indeed a straightforward count analogy of the Gaussian model where the exponential map plays the same role as the inverse link function in log-linear models. 

For the remainder of this manuscript, we assume that we have observed a random sample $X_1, \ldots, X_n$ from the model \eqref{eq:matrix_poisson_model_conditional}. Our objectives are then three-fold: (i) We first derive root-$n$ consistent estimates for the model parameters. Of these, especially of interest are the loading matrices $U_1, U_2$ which describe how the elements of the observations $X_i$ depend on the corresponding latent variables in $Z_i$. (ii) Given the parameter estimates, we establish estimators for the latent dimensions $d_1, d_2$. The estimators are based on a recent idea for using predictor augmentation for rank estimation \citep{luo2020order}. (iii) Finally, given the previous information, we estimate the latent matrices $Z_i$ themselves. Note that the estimates are again necessarily noisy as error is introduced to the model \eqref{eq:matrix_poisson_model_conditional} through the Poisson sampling.

In the special case where $p_2 = 1$ (and our observations are vectors), model \eqref{eq:matrix_poisson_model_conditional} reduces to a multivariate Poisson log-normal distribution (meaning that the observations are conditionally Poisson-variate with log-normal mean parameters), that was first proposed in \cite{aitchison1989multivariate} for modelling multivariate count data. However, they did not consider the model from the viewpoint of dimension reduction, meaning that our results on the estimation of the latent variables and their dimension are novel also for the case $p_2 = 1$. 

\subsection{Parameter estimation}\label{subsec:parameter_estimation}

The model \eqref{eq:matrix_poisson_model_conditional} has a total of six parameters to estimate, $\mu$, $ U_1 $, $U_2$, $\tau^2$, $\Lambda_1$ and $\Lambda_2$. The model being fully parametric, a natural approach to their estimation would be maximum likelihood, as was done with the vectorial version of the model in \cite{aitchison1989multivariate, izsak2008maximum, hall2011theory}. However, the marginal density of $X$ in the model involves an integral lacking a closed-form solution which, besides complicating the parameter estimation, would also make studying the asymptotic properties of the estimators very difficult. Hence, we base our subsequent estimators on the method of moments which, conveniently, yields analytical solutions with tractable asymptotic behavior. Note that the vectorial counterparts of the method of moments estimators were considered already in \cite{aitchison1989multivariate,izsak2008maximum} as initial values for the MLEs.


For $j, k = 1, \ldots , p_1$, $j \neq k$, we define the following quantities
\begin{align}\label{eq:left_side_s_elements} 
	s_{1, jk} := \frac{1}{p_2} \sum_{\ell= 1}^{p_2} \log \left\{ \frac{\mathrm{E}(x_{j\ell} x_{k \ell})}{\mathrm{E}(x_{j\ell}) \mathrm{E}(x_{k \ell})} \right\}, \quad	s_{1, jj} := \frac{1}{p_2} \sum_{\ell = 1}^{p_2} \log \left[ \frac{\mathrm{E} \{ x_{j \ell} (x_{j \ell} - 1) \} }{ \{ \mathrm{E}(x_{j \ell}) \}^2} \right],
\end{align}
along with their ``right-hand side'' variants, defined for $j, k = 1, \ldots , p_2$, $j \neq k$,
\begin{align*} 
	s_{2, jk} := \frac{1}{p_1} \sum_{\ell= 1}^{p_1} \log \left\{ \frac{\mathrm{E}(x_{\ell j} x_{\ell k})}{\mathrm{E}(x_{\ell j}) \mathrm{E}(x_{\ell k})} \right\}, \quad	s_{2, jj} := \frac{1}{p_1} \sum_{\ell = 1}^{p_1} \log \left[ \frac{\mathrm{E} \{ x_{\ell j} (x_{\ell j} - 1) \} }{ \{ \mathrm{E}(x_{\ell j}) \}^2} \right].
\end{align*}

Let $S_1$ be the $p_1 \times p_1$ matrix having the $s_{1, jk}$ as its elements and analogously for $S_2$. Recall also that $\tau > 0$ denotes the joint scaling parameter of the latent covariance matrices in the model \eqref{eq:matrix_poisson_model_conditional}. Then the following holds.

\begin{lemma}\label{lem:covariance_estimators}
Under the model \eqref{eq:matrix_poisson_model_conditional}, we have
\begin{align*}
    S_1 = \tau^2 U_1 \Lambda_1 U_1', \quad S_2 = \tau^2 U_2 \Lambda_2 U_2'.
\end{align*}
\end{lemma}

Lemma \ref{lem:covariance_estimators} shows that $ \mathrm{tr}(S_1)/(2 p_1) + \mathrm{tr}(S_2)/(2 p_2) = \tau^2 $, allowing the estimation of $\tau^2$ through $S_1$ and $S_2$. Consequently, the matrices $U_1$ and $\Lambda_1$ can be estimated through the leading $d_1$ eigenvectors and eigenvalues of $S_1/\tau^2$, respectively (implying that $S_1$ plays the role of the matrix $\mathrm{Cov}_1(X)$ in the Poisson model), and $U_2, \Lambda_2$ can be obtained similarly from $S_2$. This leaves us just with the mean parameter $\mu$ which, while a nuisance parameter in the Gaussian model \eqref{eq:matrix_normal_model}, may in the Poisson model be of independent interest, being part of the latent variables $\mu + U_1 Z U_2'$. To estimate it, we use the relationship
\begin{align*}
    \mu_{j k} = 2 \log \mathrm{E}(x_{jk}) - \frac{1}{2} \log \mathrm{E} \{ x_{j k} (x_{j k} - 1) \}
\end{align*}
following from the proof of Lemma \ref{lem:covariance_estimators}.

Let us next interpret the matrices $S_1$ and $S_2$. In the special case when $p_2 = 1$ (making the observation $X$ simply a $p_1$-variate vector $x$), the matrix $S_1$ has its $(j, k)$th off-diagonal element and its $(j, j)$th diagonal element equal to
\begin{align}\label{eq:S_for_vectors}
    \log \left\{ \frac{\mathrm{E}(x_{j} x_{k })}{\mathrm{E}(x_{j}) \mathrm{E}(x_{k})} \right\} \quad \mbox{and} \quad \log \left[ \frac{\mathrm{E} \{ x_{j} (x_{j} - 1) \} }{ \{ \mathrm{E}(x_{j}) \}^2} \right]
\end{align}
respectively, showing that $S_1$ may be viewed as a count data analogue of the ordinary covariance matrix. That is, instead of additive centering by the mean, we conduct the multiplicative centering $x_j \mapsto x_j/\mathrm{E}(x_j)$ and, instead of raw moments, we use the factorial moments. In the literature on elliptical distributions, a commonly used family of alternatives to the covariance matrix are known as scatter functionals, defined as any \textit{affine equivariant} mappings $F \mapsto S(F)$ of a $p$-variate distribution $F$ to the space of positive semi-definite matrices. By affine equivariance, it is meant that, for any invertible matrix $A \in \mathbb{R}^{p \times p}$ and any $b \in \mathbb{R}^p$, the scatter functional satisfies $S(F_{A, b}) = A S(F) A'$ where $F_{A, b}$ is the distribution of the random vector $A x + b$ and $x \sim F$, see, for example, \cite{tyler2009invariant} for examples and references on scatter functionals in the context of dimension reduction.

Now, being also an alternative of sorts to the covariance matrix, it is of interest to see whether the current matrix $S_1$ possesses any similar properties. For a random $p$-variate count vector $x$, it is seen from \eqref{eq:S_for_vectors} that the matrix $S_1$ is invariant under the transformations $x \mapsto D x$ where $D \in \mathbb{R}^{p \times p}$ is an arbitrary diagonal matrix with positive diagonal elements. Hence, we observe that $S_1$ does not actually measure the ``scatter'', or scale, of $x$ but rather some higher order property (``shape''). Inspection also reveals that if the $j$th and $k$th element of $x$ are independent, the corresponding off-diagonal element of $S_1$ vanishes, which is known in the context of scatter functionals as the (element-wise) independence property \citep{nordhausen2015cautionary}. Take now the elements of $x$ to be i.i.d. from various standard count data distributions: If $x_1 \sim \mathrm{Po}(\lambda)$, we have $s_{1, 11} = 0$ (and, consequently, $S_1 = 0)$, showing that Poisson-distribution is viewed as being pure noise by the matrix $S_1$ (this observation will be used in Section \ref{subsec:order_determination} to estimate the latent dimension $d$). If $x_1 \sim \mathrm{NegBin}(r, p)$ (the negative binomial with success probability $p$ and stopping after the $r$th failure), then $s_{1, 11} = \log(1 + 1/r)$. If $x \sim \mathrm{Bin}(n, p)$, we get $s_{1, 11} = \mathrm{log}(1 - 1/n)$, showing, in particular, that $S_1$ is not necessarily positive semi-definite. A common thread behind the previous cases is that in all three the value of $S_1$ is independent of a subset of the involved parameters (taken to the extreme with the Poisson-distribution). Additionally, the signs of the diagonal elements of $S_1$ correspond in each case with the presence of overdispersion in the distributions (negative/positive sign being linked with underdispersion/overdispersion). Hence, instead of being a measure of scale, it seems more fitting to view the matrix $S_1$ as measuring the amount of overdispersion in the data (this analogy is not perfect, as, e.g., for the negative binomial distribution the severity of the overdispersion depends on the parameter $p$, to which $S_1$ is invariant).

With the interpretation out of the way, we next turn to the asymptotic properties of the parameter estimates under the model \eqref{eq:matrix_poisson_model_conditional}. For simplicity, we state the results for the aggregate parameters $S_1$ and $S_2$, after which consistent estimates for, e.g., $U_1, \Lambda_1$ are obtained from the eigendecomposition of $S_1$, the consistency of the former requiring that the diagonal elements of $\Lambda_1$ are distinct. Regardless of the multiplicities of the signal eigenvalues in $\Lambda_1$ and $\Lambda_2$, the projection matrices $U_1 U_1'$ and $U_2 U_2'$ to the latent subspaces are in any case unique and consistently estimable through the corresponding sample versions, see, e.g., \cite{tyler1981asymptotic,eaton1991wielandt} for large-sample properties of eigenvectors and eigenvalues. Given a random sample $X_1, \ldots, X_n$ from the model \eqref{eq:matrix_poisson_model_conditional}, let in the following $S_{n1}$ denote the sample version of the matrix $S_1$ (that is, with elements as in \eqref{eq:left_side_s_elements} but with the expected values replaced by sample means) and define $S_{n2}$ analogously. Finally, define $t^2_n := \mathrm{tr}(S_{n1})/(2 p_1) + \mathrm{tr}(S_{n2})/(2 p_2) $ and let the elements of the $p_1 \times p_2$ matrix $M_n$ be $m_{n, jk} := 2 \log \{ (1/n) \sum_{i=1}^n x_{i, jk} \}  - (1/2) \log \{ (1/n) \sum_{i=1}^n x_{i, jk} (x_{i, jk} - 1) \}$. Then, the previous constitute root-$n$ consistent estimators of the model parameters.

\begin{theorem}\label{theo:asymptotics}
    Under the model \eqref{eq:matrix_poisson_model_conditional}, each of the deviations,
    \begin{align*}
        S_{n1} - S_1, \quad S_{n2} - S_2, \quad t_n^2 - \tau^2 \quad \mbox{and} \quad M_n - \mu,
    \end{align*}
    is of the order $\mathcal{O}_p(1/\sqrt{n})$ as $n \rightarrow \infty$.
\end{theorem}

Theorem \ref{theo:asymptotics} is an understatement in the sense that, besides the convergence rate, its proof would also allow establishing the exact limiting normal distributions of each of the estimators, with closed-form solutions for the corresponding asymptotic covariance matrices. However, we have refrained from presenting them here for two reasons. First, the limiting covariance matrices play no role in any of the subsequent developments, nor would they allow comparison to any competing methods as we are unaware of any equivalent results for the existing count data dimension reduction methods in the literature. Second, the asymptotic covariance matrices of the estimators are algebraically very complicated and lacking of any clear, intuitive form. To exemplify the latter point, in Appendix \ref{sec:limiting_distribution} we have derived the asymptotic variance of $S_{n1}$ in the simple special case with $p_1 = p_2 = 1$. The previous considerations aside, the result of Theorem \ref{theo:asymptotics} is, with or without exact limiting distributions, essential to our development in that it allows the estimation of the latent variables, $Z_1, \ldots , Z_n$, described next.

\subsection{Latent component estimation}\label{subsec:variable_estimation}

For convenience, the results of this section are formulated in terms of the column vectorizations $ x $ and $ z $ of the matrices $X$ and $Z$. Additionally, we denote $m := \mathrm{vec}(\mu)$, $\Lambda := \Lambda_2 \otimes \Lambda_1$, $U := U_2 \otimes U_1$, $p := p_1 p_2$ and $d := d_1 d_2$. Recall from Section \ref{subsec:low_rank_models} that in the matrix normal model \eqref{eq:matrix_normal_model_conditional}, a natural estimator for the latent variables, or principal components (PCs), $z$ is obtained straightforwardly as $U' \{ x - \mathrm{vec}(\mu) \}$. The simple, linear form of the estimator can be seen to follow from the fact that the observed and the latent matrices belong to the same distributional family, a property that does not hold for the Poisson model \eqref{eq:matrix_poisson_model_conditional}. However, we can still draw an analogy with the normal model by observing that the linear estimate $U' \{ x - m \}$ admits a characterization as the mode of the conditional distribution of the scaled latent components $(I_d + \sigma^2 \tau^{-2} \Lambda^{-1}) z$ given the observation $x$.

\begin{lemma}\label{lem:normal_conditional_distribution}
    Under the normal model \eqref{eq:matrix_normal_model_conditional}, we have
    \begin{align*}
        (I_d + \sigma^2 \tau^{-2} \Lambda^{-1}) z \mid x \sim \mathcal{N}_d \{ U'  (x - m), \sigma^2 (I_d + \sigma^2 \tau^{-2} \Lambda^{-1}) \}.
    \end{align*}
\end{lemma}

Guided by Lemma \ref{lem:normal_conditional_distribution}, we estimate the principal components $z$ in the Poisson model \eqref{eq:matrix_poisson_model_conditional} analogously as the mode of the conditional distribution of $ z $ given $ x $ (we do not incorporate the scaling matrix $I_d + \sigma^2 \tau^{-2} \Lambda^{-1}$ as $\sigma^2$ has no analogue in the Poisson model and, besides, scale is often anyway seen as a nuisance in dimension reduction). Unlike in the normal model, the resulting conditional distribution does not belong to any standard distributional family, but its mode can still be estimated efficiently through numerical maximization of a concave objective function, as shown in the next lemma. Similar approaches have been used earlier for count data in \cite{li2010simple, kenney2019poisson}. In the sequel, we denote by $ \ell(z | x) := \log f_{z | x} (z | x)$ the logarithmic density function of the conditional distribution.

\begin{theorem} \label{theo:logarithmic_conditional_density}
The logarithmic conditional density $\ell(z | x)$ satisfies the following.
\begin{itemize}
    \item[i)] For a constant $C$ not depending on $z$,
	\begin{align*}
	\ell(z | x) = C + x' U z - 1' \exp ( m + U z ) - \frac{1}{2 \tau^2} z' \Lambda^{-1} z,
	\end{align*}
	where $1 \in \mathbb{R}^p$ is a vector of ones and the exponential function is applied element-wise.
	\item[ii)] For all $x \in \mathbb{R}^{p}$, the function $z \mapsto \ell (z | x)$ is strictly concave and admits a unique maximum in $\mathbb{R}^{d}$.
\end{itemize}
\end{theorem}

Denote the gradient and the Hessian matrix of the map $z \mapsto \ell(z | x)$ as $g(z | x)$ and $H(z | x)$, respectively, the exact forms of which are given in the proof of Lemma \ref{theo:logarithmic_conditional_density} in Appendix \ref{sec:proofs}. Furthermore, given the estimates of the model parameters from Section \ref{subsec:parameter_estimation}, denote by $\ell_n(z | x)$, $g_n(z | x)$ and $H_n(z | x)$ the logarithmic conditional density, gradient and Hessian, respectively, with the parameter estimates plugged in. The $d$-variate latent vector $z_i$ corresponding to a vectorized observation $x_i$ can now be estimated as the unique maximizer of $z \mapsto \ell_n(z | x_i)$ using the standard Newton-Raphson method, Lemma \ref{theo:logarithmic_conditional_density} guaranteeing its convergence. Recall finally that the model \eqref{eq:matrix_poisson_model_conditional} assumes the principal components to have zero mean. Hence, as the final step in their estimation, we still center the estimated sample PCs $z_1, \ldots , z_n \in \mathbb{R}^d$.

We end the section with a collection of remarks. Interestingly, Theorem \ref{theo:logarithmic_conditional_density} reveals that the estimate of the latent variables $z_i$ depends on the observed data $x_i$ only through the projection $U_n' x_i$ where $U_n := U_{n2} \otimes U_{n1}$ and the $p_1 \times d_1$ matrix $U_{n1}$ contains any first $d_1$ eigenvectors of $S_{n1}$ as its columns (and similarly for $U_{n2}$). This is somewhat surprising as, based on the formulation of the model \eqref{eq:matrix_poisson_model_conditional}, one would expect the matrix $U_n$ to act linearly only with $Z$, and not with $X$ (to which it has a non-linear functional dependency). Finally, while we viewed the Gaussian estimate $U' \{ x - \mathrm{vec}(\mu) \}$ above as the mode of the conditional normal distribution in Lemma \ref{lem:normal_conditional_distribution}, it is, naturally, also the mean of the same distribution. Thus, an alternative strategy in the Poisson model would be to base the estimates of the latent components on the conditional means of the random vector $z$ given the observations $x_i$. However, while equally valid (and heuristic) as the taken viewpoint, relying on the mean would lead to an intractable integral requiring numerical approximation, leading us to favor the mode approach with its concave optimization problem.



\subsection{Dimension estimation}\label{subsec:order_determination}

We next develop an estimator for the latent dimension $d_1$ using the recently proposed idea of predictor augmentation \citep{luo2020order}. By the symmetry of the model \eqref{eq:matrix_poisson_model_conditional}, an estimator for $d_2$ is obtained exactly analogously after the transposition of the observations.

In predictor augmentation, artificially generated noise is concatenated to the observations in order to reveal the cut-off point from positive values to zero in the spectrum of the matrix of interest. More precisely, given a random sample $X_1, \ldots, X_n$ from the model \eqref{eq:matrix_poisson_model_conditional}, fix a positive integer $r_1 \in \mathbb{N}^+$ and let $X_1^*, \ldots , X_n^*$ be the augmented sample where $X_i^* = (X_i', R_i')'$ and the elements of the $r_1 \times p_2$ matrices $R_i$ are sampled i.i.d. from the Poisson distribution with unit mean, $i = 1, \ldots , n$. Letting $S_{n1}^*$ denote the equivalent of the matrix $S_{n1}$ but computed from the augmented sample, techniques similar to the ones used in Lemma \ref{lem:covariance_estimators} and Theorem \ref{theo:asymptotics} show that
\begin{align}\label{eq:augmented_S1}
    S_{n1}^* = \begin{pmatrix}
    \tau^2 U_1 \Lambda_1 U_1' & 0 \\
    0 & 0
    \end{pmatrix} + \mathcal{O}_p(1/\sqrt{n}).
\end{align}
Let now the $r_1$-dimensional vectors $\beta_{n11}, \ldots , \beta_{n1(p_1 + r_1)}$ contain the final $r_1$ elements of any set of orthogonal eigenvectors of $S_{n1}^*$ (that is, $\beta_{n11}$ contains the last $r_1$ entries of an eigenvector corresponding to the first eigenvalue of $S_{n1}^*$ etc.) Now, for $k \leq d_1$, we expect the norms $\| \beta_{n1k} \|$ to be close to zero as the corresponding eigenspaces are, in the limit $n \rightarrow \infty$, concentrated fully on the subspace spanned by the $d_1$ columns of the $(p_1 + r_1) \times d_1$ matrix $(U_1', 0)'$. On the other hand, for $k > d_1$, there is no reason for $\| \beta_{n1k} \|$ to be small as the final $p_1 + r_1 - d_1$ eigenvalues of the limiting matrix in \eqref{eq:augmented_S1} are all equal to zero and, hence, the corresponding eigenvectors should not favor any direction (in the null space). For a rigorous presentation of this concept, along with more details on the full procedure, see \cite{luo2020order}. In predictor augmentation, this information provided by the eigenvectors is further supplemented by the eigenvalues $\lambda_{n11} \geq \cdots \geq \lambda_{n1(p_1 + r_1)}$ of the matrix $S_{n1}^*$ to define the objective function $\phi_{n1}: \{0, \ldots , p_1\} \to \mathbb{R}$
\begin{align}\label{eq:augmentation_objective}
    \phi_{n1}(k) = \sum_{j = 0}^k \| \beta_{n1j} \|^2 + \frac{\lambda_{n1(k + 1)}}{1 + \sum_{j = 1}^{k + 1} \lambda_{n1j}},
\end{align}
where we define $\| \beta_{n10} \| $ to be equal to zero. Note that the second term of \eqref{eq:augmentation_objective} corresponds essentially to a scaled version of the scree plot used commonly in PCA. By the earlier discussion, we expect the first term of $\phi_{n1}(k)$ to be small for $k \leq d_1 $, whereas, the second term takes (for large enough $n$) small values for $k \geq d_1$ (i.e., at the indices corresponding to the zero limit eigenvalues). Consequently, we take as our estimate of $d_1$ the value $k$ at which $\phi_{n1}$ is minimized,
\begin{align*}
    d_{n1} := \mbox{argmin}_{k \in \{ 0, \ldots , p \} } \phi_{n1}(k). 
\end{align*}
To increase the stability of the estimate for small $n$, \cite{luo2020order} further advocated independently carrying out the augmentation procedure $s_1$ times and replacing $\| \beta_{n1j} \|^2$ in \eqref{eq:augmentation_objective} with its mean over the $s_1$ replicates. Similarly, we also replace the eigenvalues $\lambda_{n1j}$ with their means over the replicates (although this was not done in \cite{luo2020order}). The procedure has two tuning parameters, $r_1$ and $s_1$, the latter of which directly reduces variation in the results and, hence, we suggest to use large values for it in practice. Suggestions on choosing the value of $r_1$ are formulated later in Section \ref{sec:examples}.

Our simulation results in Section \ref{sec:examples} suggest that $d_{n1}$ could be a consistent estimator of $d_1$ as $n \rightarrow \infty$, but proving this turned out to be less than straightforward. Namely, \cite{luo2020order} give sufficient conditions under which estimators such as $d_{n1}$ are consistent for the true dimension. Our scenario is easily checked to satisfy these assumptions apart from one, i.e., the requirement (12) in \cite{luo2020order} that the sequence of augmented matrices $S_{n1}^*$ is in a specific sense contiguous to Lebesgue measure. Now, a standard way of showing this would be to establish that $\sqrt{n} (S_{n1}^* - S_{1}^*)$, where $S_{1}^*$ is the limit of $S_{n1}^*$, converges in total variation to a non-singular normal distribution. However, by the classical result of Prohorov, see, e.g., Theorem 2.2 in \cite{bally2016asymptotic}, the convergence in total variation happens in the central limit theorem if and only if the corresponding sample estimators have non-trivial absolutely continuous components. Naturally, this is not the case with our count data model, implying that alternative strategies must be sought and, hence, we leave this question for future study.



\begin{figure}
    \centering
    \includegraphics[width = 1\textwidth]{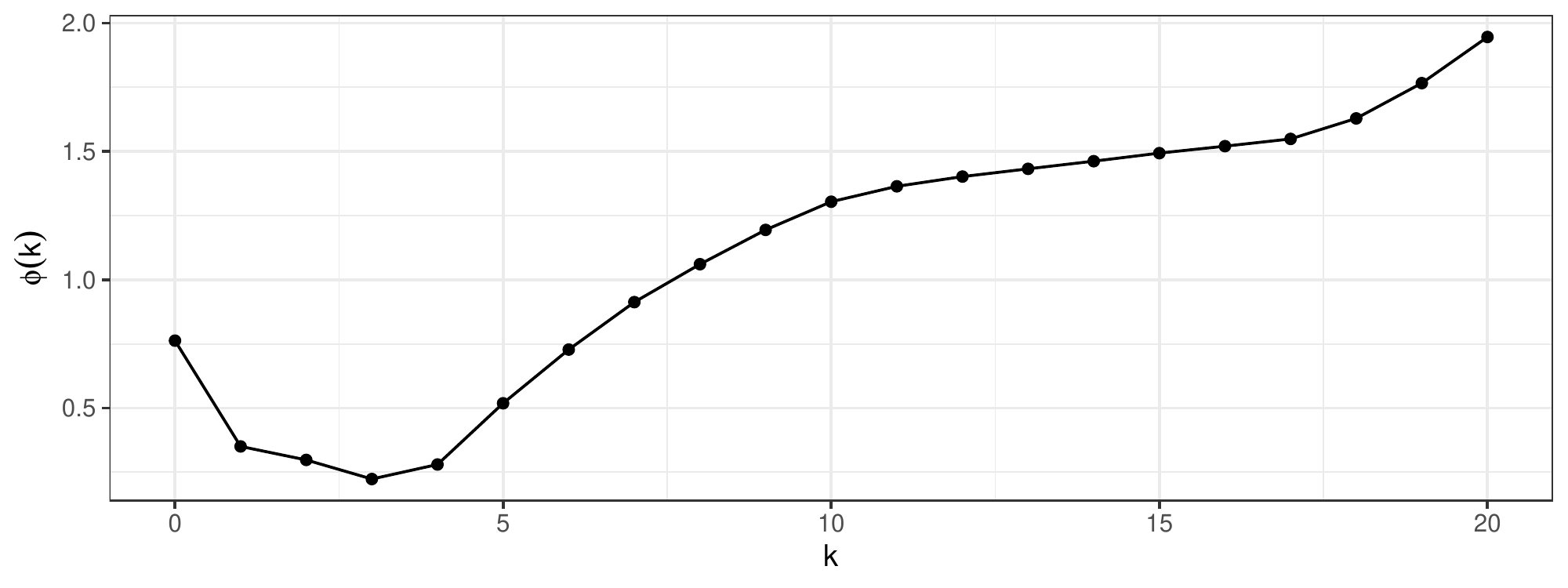}
    \caption{Plot of the map $k \mapsto \phi_{n1}(k)$ for the abundance data set. The minimum value is achieved at $k = 3$. }
    \label{fig:abundance_1}
\end{figure}

We next illustrate the dimension determination procedure in a simple special case of the model \eqref{eq:matrix_poisson_model_conditional} with $p_2 = 1$, meaning that the observations can be treated as $p_1$-dimensional vectors and it is sufficient to consider dimension estimation for the left-hand side of the model only (recall that our results are novel also in this special case). This, and all the examples to follow, were implemented in the language R \citep{Rbase}. The data set \texttt{microbialdata} in the R-package \texttt{gllvm} \citep{r_gllvm} consists of the abundances of $985$ bacteria species measured at $n = 56$ soil sample sites located either in Austria, Finland or Norway. For the purposes of this demonstration, we limit our attention to the subset of the $p_1 = 20$ most abundant species, defined as the ones having the least proportions of zero counts over all 56 sites. We now apply the proposed dimension determination procedure to this subset of the data with the choices $r_1 = \lceil p_1/5 \rceil = 4$ and $s_1 = 100$. The value of $r_1$ was chosen for its good success in the simulation results of \cite{luo2020order} and for $s_1$ we chose a large enough value to increase the stability of the estimate in the presence of the rather low sample size. The resulting function $\phi_{1n}$ is plotted in Figure \ref{fig:abundance_1} and leads to the estimate $d_{n1} = 3$. To assess whether the corresponding latent variables are meaningful, we estimate the model parameters, followed by the values of the first three latent variables for each of the $n = 56$ sites using the method of Section~\ref{subsec:variable_estimation}. For ease of presentation, we limit ourselves to the first two latent variables whose scatter plot is presented in Figure \ref{fig:abundance_2} (the colored plot markers), overlaid with the loadings of the first two latent variables, i.e., the first two columns of $U_{n1}$ (the numbers connected with dashed lines to the origin). The numbering of the $p_1 = 20$ species corresponds to their indexing in a specific taxonomy. The sites of the three regions appear to be rather well-separated in the first two latent variables and, to verify this finding, we fit a generalized linear Poisson latent variable model \citep{niku2017generalized} between the abundances and the region variable using the function \texttt{gllvm} in the R-package \texttt{gllvm}. Based on the coefficients estimates of the model, the sites in Austria are the most (the least) associated with the bacteria species 8, 4 (52, 184), the sites in Finland are the most (the least) associated with the species 7, 1 (4, 13) and the sites in Norway are the most (the least) associated with the species 64, 1242 (8, 70). Comparison of the previous with Figure \ref{fig:abundance_2} now reveals that the same pattern is indeed rather accurately reflected in the positioning of the loadings and the sites in our ``biplot'', showing that the latent variables managed to capture essential biological information. Further illustration of the methodology in the general case $p_2 > 1$ are given in Section~\ref{sec:examples}.

\begin{figure}
    \centering
    \includegraphics[width = 1\textwidth]{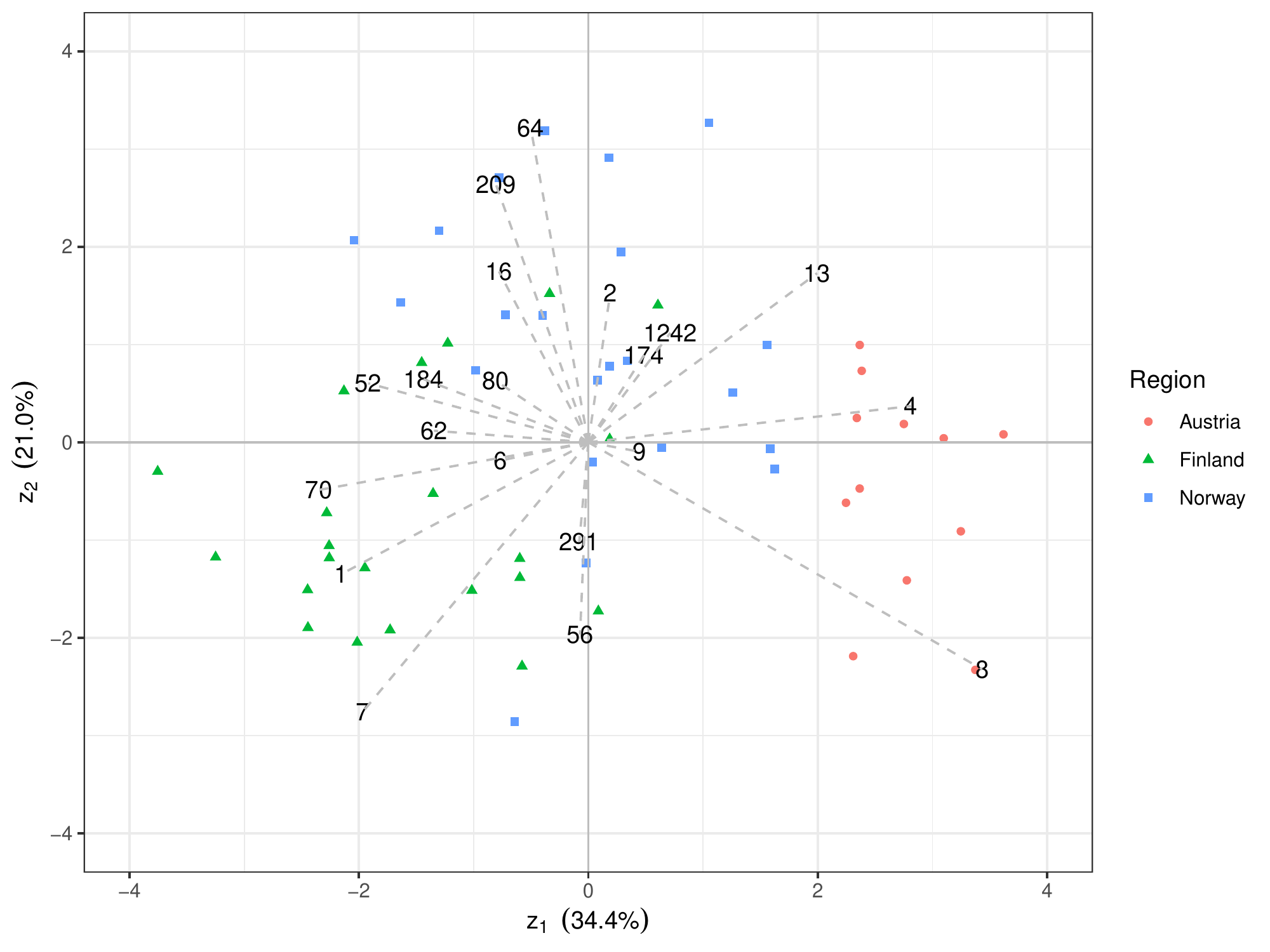}
    \caption{Scatter plot of the first two latent components for the abundance data set (the colored plot markers), overlaid with the corresponding loadings (the numbers connected with dashed lines to the origin). The numbering of the species represents their indexing in a specific taxonomy. The percentages on the axes denote the corresponding explained proportions of variance (ratios of the individual eigenvalues of $S_{n1}$ to their sum).}
    \label{fig:abundance_2}
\end{figure}

\section{Examples}\label{sec:examples}

\subsection{Dimension estimation in a simulation}

We next study the performance of the augmentation procedure in estimating the latent dimension. Simulations exploring the efficiency of our parameter estimators will be conducted later in Section \ref{sec:zero_inflated} once we have defined the zero-inflated version of the model. We take as our competitor the same augmentation estimator but applied to the Gaussian model \eqref{eq:matrix_normal_model}. This is essentially achieved by replacing the matrix $S_1$ in the augmentation procedure with $\mathrm{Cov}_1(X)$, see Section \ref{subsec:low_rank_models}, and similarly for $S_2$. This estimator, which can be seen as a ``naive'' data type-ignoring approach to matrix-valued count data has been recently studied in the context of image data in \cite{radojicic2021estimating}, see their work for more details.

We use two different sample sizes, $n = 100, 500$, and two different observed dimensions, either $(p_1, p_2) = (10, 5)$ (``Low dimension'') or $(p_1, p_2) = (50, 25)$ (``High dimension''). Two different models for the covariance parameters are considered: In {Model 1}, the first dimension is of rank one, $S_1 = \tau^2 U_1 \Lambda_1 U_1' = 1_{p_1} 1_{p_1}'$ ($1_{p_1}$ is the $p_1$-dimensional vector with all elements equal to one), and the second dimension is of rank five, $S_2 = \tau^2 U_2 \Lambda_2 U_2' = W E_5 W$, where $W$ is a uniformly random $p_2 \times p_2$ orthogonal matrix (drawn separately for every iteration of the study) and $E_5 \in \mathbb{R}^{p_2 \times p_2}$ is a diagonal matrix with its first five diagonal elements equal to one and the rest of them zero. In {Model 2}, both dimensions are taken to have rank five, with the covariance parameters being created analogously to that of the second dimension in Model 1. Data from every combination of the parameters and models are simulated 200 times and for each replicate we estimate the two dimensions with seven different approaches. These include our proposed augmentation approach with the numbers of augmentations $(r_1, r_2) = (p_1, p_2), (\lceil p_1/2 \rceil, \lceil p_2/2 \rceil), (\lceil p_1/5 \rceil, \lceil p_2/5 \rceil), (\lceil p_1/10 \rceil, \lceil p_2/10 \rceil), (1, 1)$ (denoted in the following by A1, A2, A3, A4, A5 respectively), where in each case we take the numbers of repetitions to be $s_1 = s_2 = 5$. In addition, we consider the Gaussian augmentation procedure as described in \cite{radojicic2021estimating} and implemented in the R-package \texttt{tensorBSS} \citep{rtensorBSS}, and having either $r = 1$ or $r = 5$ augmentations (denoted in the following by G1, G2), with both cases using $s = 5$ repetitions.

\begin{table}[t]
\small
\centering
\caption{Results of the dimension estimation study. The numbers refer to the percentages of correctly estimated dimensions in the different combinations of models and parameters. The results for the left-hand side dimension $d_1$ are denoted by L whereas R signifies the estimates for the right-hand side dimension~$d_2$.}
\label{tab:simu_2}
\begin{tabular}{c|c|c|p{0.44cm}p{0.44cm}|p{0.44cm}p{0.44cm}|p{0.44cm}p{0.44cm}|p{0.44cm}p{0.44cm}|p{0.44cm}p{0.44cm}|p{0.44cm}p{0.44cm}|p{0.44cm}p{0.44cm}}
  &  &  & \multicolumn{2}{|c|}{G1} & \multicolumn{2}{|c|}{G2} & \multicolumn{2}{|c|}{A1} & \multicolumn{2}{|c|}{A2} & \multicolumn{2}{|c|}{A3} & \multicolumn{2}{|c|}{A4} & \multicolumn{2}{|c|}{A5} \\
  \hline
 Dim. & Model & $n$ & L & R & L & R & L & R & L & R & L & R & L & R & L & R \\ 
  \hline
 \multirow{2}{*}{Low} & \multirow{2}{*}{1} & 100 & 98 & 96 & 100 & 83 & 100 & 76 & 100 & 95 & 100 & 100 & 100 & 100 & 100 & 100 \\ 
  &  & 500 & 100 & 100 & 100 & 100 & 100 & 100 & 100 & 100 & 100 & 100 & 100 & 100 & 100 & 100 \\
 \hline
 \multirow{2}{*}{Low} & \multirow{2}{*}{2} & 100 & 26 & 0 & 0 & 0 & 64 & 76 & 100 & 99 & 100 & 100 & 98 & 100 & 100 & 100 \\ 
  &  & 500 & 84 & 0 & 46 & 0 & 100 & 100 & 100 & 100 & 100 & 100 & 100 & 100 & 100 & 100 \\
 \hline
 \multirow{2}{*}{High} & \multirow{2}{*}{1} & 100 & 87 & 23 & 100 & 86 & 100 & 100 & 100 & 100 & 100 & 100 & 100 & 100 & 100 & 99 \\ 
  &  & 500 & 98 & 2 & 100 & 34 & 100 & 100 & 100 & 100 & 100 & 100 & 100 & 100 & 100 & 100 \\
 \hline
 \multirow{2}{*}{High} & \multirow{2}{*}{2} & 100 & 10 & 35 & 8 & 0 & 0 & 0 & 0 & 0 & 0 & 0 & 0 & 0 & 38 & 8 \\ 
  &  & 500 & 58 & 88 & 100 & 98 & 0 & 0 & 0 & 0 & 32 & 0 & 98 & 44 & 90 & 95 \\ 
   \hline
\end{tabular}
\end{table}

The rounded percentages of correctly estimated left and right dimensions over the 200 replicates are shown in Table \ref{tab:simu_2} where L and R refer to the left-hand side dimension $d_1$ and the right-hand side dimension $d_2$, respectively. Comparison of the methods A1 -- A5 shows that, almost without exception, using a smaller amount of augmentations $(r_1, r_2)$ is beneficial, an observation that matches with the results of the simulation study in \cite{luo2020order} (in the context of vectorial dimension reduction). None of A1--A5 is able to consistently estimate the dimensions in the high-dimensional version of Model 2 for $n = 100$, but, after increasing the sample size to $n = 500$, A5 achieved almost perfect results, making it the most consistent of the methods. Based on this, we suggest using the tuning parameter values $(r_1, r_2) = (1, 1)$ in practice. Also, interestingly, for $n = 100$, the method A5 has more difficulties in estimating $d_2$ (R) than $d_1$ (L) even though we have $p_1 > p_2$ and $d_1 = d_2$. Turning our attention to the Gaussian augmentations G1, G2, we observe that they work very consistently under some settings (low-dimensional Model 1) and badly underperform under some (low-dimensional Model~2). But most interestingly, G2 achieves the best performance out of all seven methods in the high-dimensional Model 2 with $n = 500$. Thus, while the Gaussian approach appears to be too unreliable to be used in practical situations of count data, it clearly does work extremely well in some specific situations, warranting more research in the future.

\subsection{Real data example}

Following in the spirit of our preliminary example in Section \ref{subsec:order_determination}, we next apply the proposed method to matrix-valued abundance data used earlier in \cite{frelat2017community}. The data consists of the relative abundances (rounded to nearest integer) of a total of $n = 65$ fish species in seven different so-called roundfish areas (RA 1 -- RA 7) in the North Sea, studied during the years 1985-2015 which we further divided into 6 time periods ($1985-1989, \ldots , 2005-2009, 2010-2015$). Thus, for the $i$th species, we have the $7 \times 6$ matrix $X_i$ whose $(j, k)$th element tells the relative abundance of that particular species in the area RA $j$ during the $k$th period.

In \cite{frelat2017community}, six biologically meaningful clusters (\textit{Southern}, \textit{Northern}, \textit{NW Increasing}, \textit{SE Increasing}, \textit{Increasing} and \textit{Decreasing}) were identified among the 65 species in the data using the combination of principal tensor analysis \citep{leibovici1998singular} and hierarchical clustering. As one of the primary practical objectives of dimension reduction is the discovery of structure (such as groups) in data, it seems reasonable to require that any successful method for reducing the dimension of the current data should be able to detect the previous six clusters, that were indeed in \cite{frelat2017community} deemed biologically internally consistent.

\begin{figure}[t]
    \centering
    \includegraphics[width = 1\textwidth]{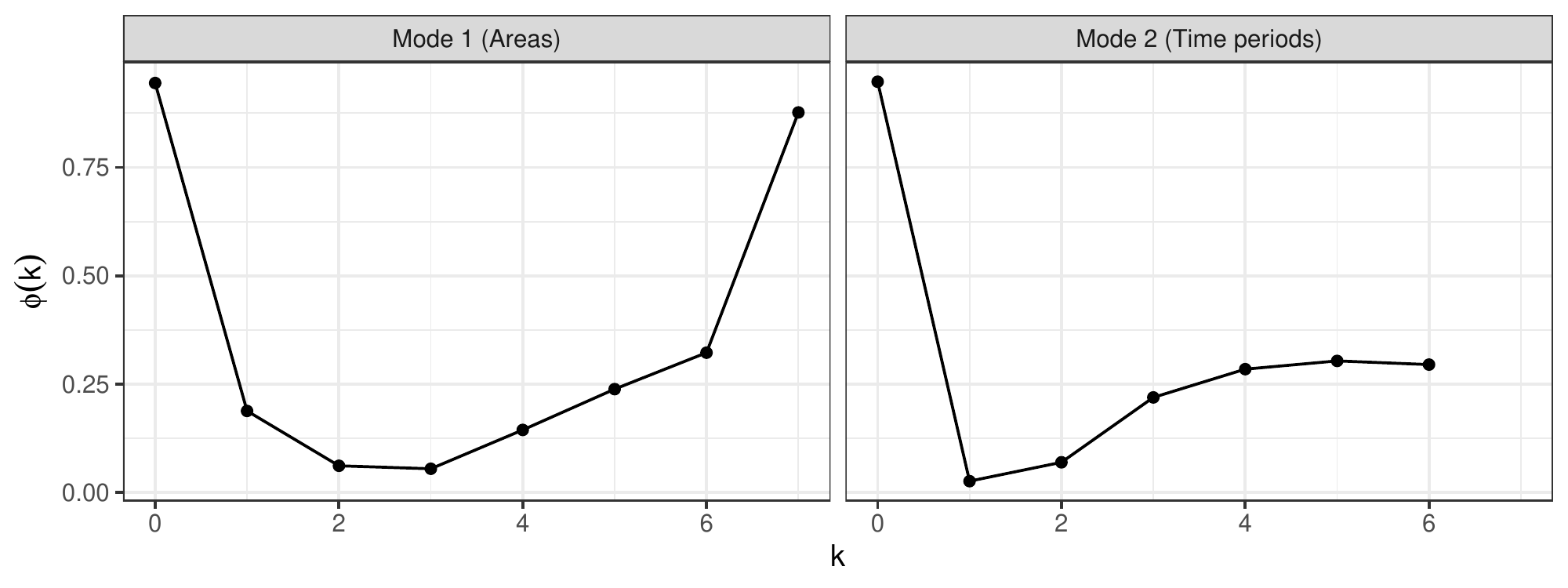}
    \caption{The two augmentation estimator curves for the matrix-valued abundance data. The left panel corresponds to the area dimension and the right one to the time dimension.}
    \label{fig:abundance_3}
\end{figure}

With the previous in mind, we next estimate our Poisson model for the data, starting with the dimension estimate curves produced by the augmentation procedure of Section \ref{subsec:order_determination} and shown in Figure \ref{fig:abundance_3} (based on the results of the simulation study, we used $ (r_1, r_2) = (1, 1)$ and $(s_1, s_2) = (100, 100)$). Based on the two plots (the left panel corresponding to the areas and the right panel to the time periods), we conclude that the latent dimension of the time periods is clearly one. For the areas, while the minimum of the curve is achieved at three dimensions, also two seems to be a reasonable option. In order not to lose any information, we retain a total of three principal components, $z_{i,11}, z_{i,21}, z_{i,31}$, estimated with the algorithm in Section \ref{subsec:variable_estimation} (of which three observations failed to converge and are not shown in the subsequent plots). Examination of the corresponding loading vectors (the columns of $U_{n1}$ and $U_{n2}$) then reveals that in both modes the first loading vector has roughly constant elements, indicating that the corresponding PC $z_{i, 11}$ simply measures the overall abundances of the species (the absolute correlation between $z_{i, 11}$ and the average abundances of the species in the 42 area-time combinations is $0.55$).


\begin{figure}[t]
    \centering
    \includegraphics[width = 1\textwidth]{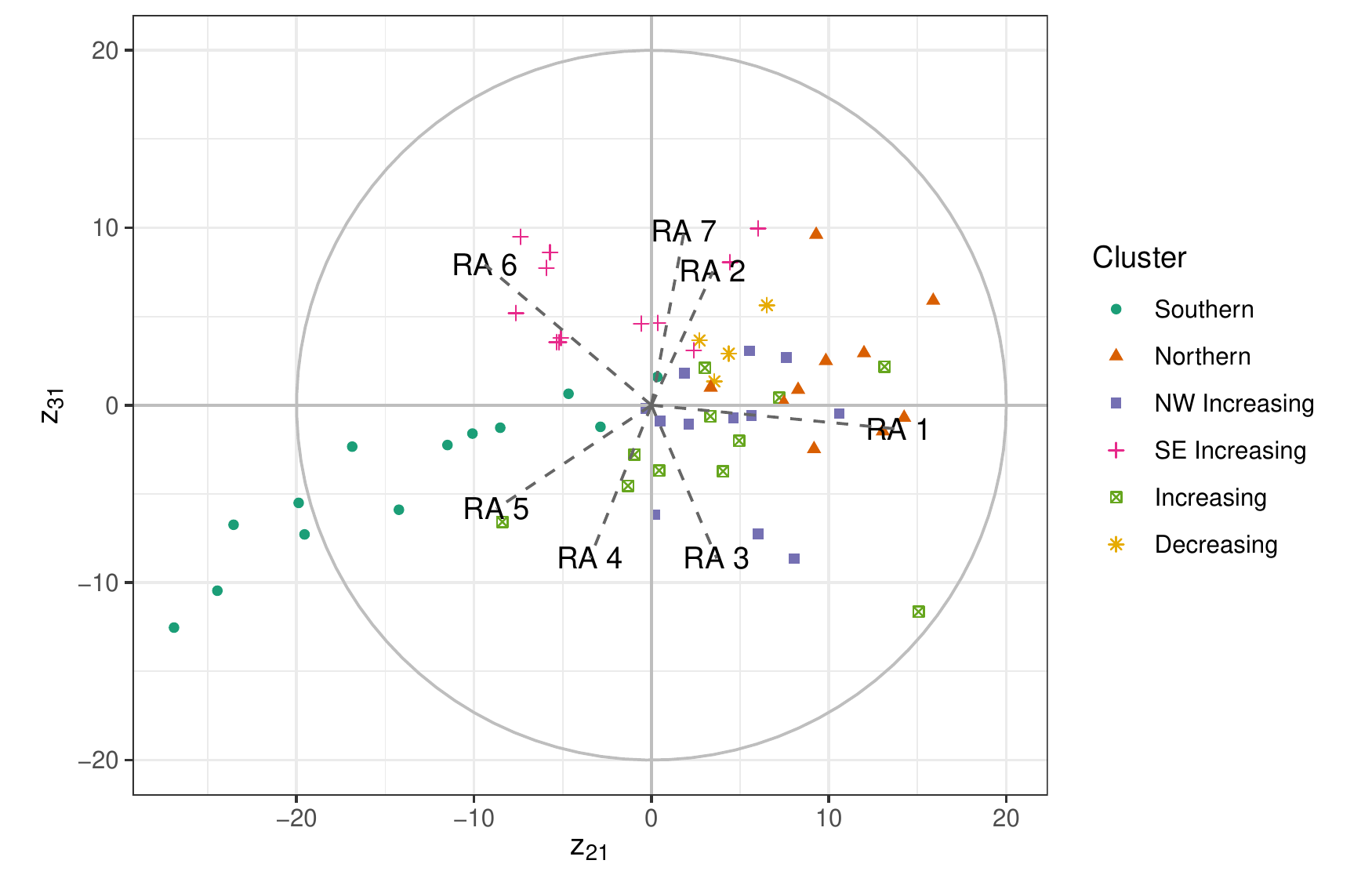}
    \caption{The scatter plot of the principal components $z_{i, 21}$ and $z_{i, 31}$ for the matrix-valued abundance data. Overlaid as dashed lines are the corresponding loadings of the seven roundfish areas and the coloring corresponds to the clustering in \cite{frelat2017community}}.
    \label{fig:abundance_5}
\end{figure}

As our interests lie deeper than in the aggregate abundances, we next ignore the PC $z_{i,11}$ and plot the remaining two, $z_{i,21}$ and $z_{i,31}$, in a bivariate scatter plot, depicted in Figure~\ref{fig:abundance_5}. The coloring/shapes in the plot correspond to the six clusters identified in\cite{frelat2017community} and we observe that they are indeed well-separated in the plot, with the exception of \textit{NW Increasing} and \textit{Increasing}. \cite{frelat2017community} actually remark that the \textit{NW Increasing} is a ``very heterogeneous cluster'' and, by Figure~3 in \cite{frelat2017community}, if their hierarchical clustering had been stopped at five instead of six clusters, it is precisely the clusters \textit{NW Increasing} and \textit{Increasing} that would have been joined next. Thus, we conclude that the principal components in Figure \ref{fig:abundance_5} have quite successfully managed to capture the group structure of the data. Overlaid in Figure \ref{fig:abundance_5} as dashed lines are also the corresponding area loadings (given by the second and third column of $U_{n1}$) for the seven roundfish areas. Comparison to Figure 4 in \cite{frelat2017community} reveals that these rather accurately capture the division of the clusters in the seven areas (for example, the \textit{Southern} cluster is heavily concentrated in RA 5, as suggested by the aligning of the corresponding group and dashed line in our Figure~\ref{fig:abundance_5}). We also observe that the loadings of the areas manage to capture some geographical information, as, after reflecting w.r.t. the $x$-axis and rotating clock-wise by 90 degrees, the loading map in Figure \ref{fig:abundance_5} matches approximately with the actual map of the seven areas in the Northern sea, see Figure 1 in \cite{frelat2017community}. Note that we have not included the loadings of the time dimension in Figure~\ref{fig:abundance_5} as it is one-dimensional for the PCs (all of $z_{i, 11}$, $z_{i, 21}$, $z_{i, 31}$ have the same column coordinate). Besides, the corresponding dimension was already earlier deemed as uninteresting.

\begin{figure}
    \centering
    \includegraphics[width = 1\textwidth]{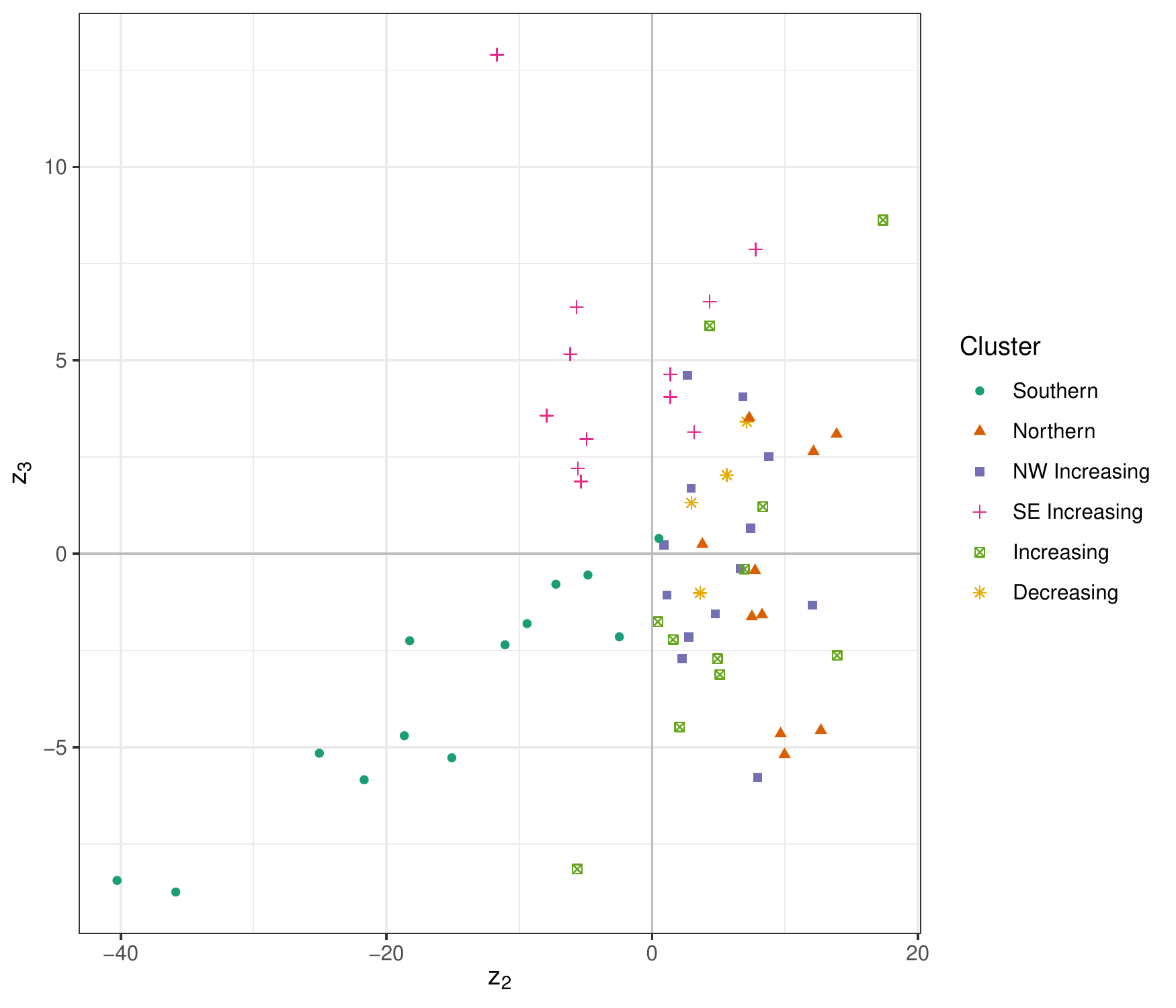}
    \caption{The scatter plot of the second and third PCs extracted from the vectorized matrix abundance data. The coloring corresponds to the clustering in \cite{frelat2017community}.}
    \label{fig:abundance_6}
\end{figure}

Before concluding, we still briefly compare the obtained results to those of two natural competitors: our proposed Poisson PCA procedure applied to vectorized observations (that is, each $7 \times 6$ matrix $X_i$ is replaced with a 42-dimensional vector $x_i$) and the popular tensor decomposition known as \textit{higher order singular value decomposition} (HOSVD) \citep{de2000multilinear} and implemented as the function \texttt{tPCA} in the R-package \texttt{tensorBSS} \citep{rtensorBSS}. The augmentation plot (not shown here) for the vectorial version of our proposed method reveals that the latent dimension is three. Similarly to the matrix model, the first PC has again almost constant loadings for all 42 variables and in Figure \ref{fig:abundance_6} we have visualized the second and the third PC. The plot clearly manages to separate the \textit{Southern} and \textit{SE Increasing} clusters but the remaining four are left more or less overlapping. In addition, incorporation of any loading information to Figure \ref{fig:abundance_6} would be difficult as we lost the distinction between the row and column variables in the vectorization. Note also that, as described in the introduction, our proposed matrix-model is actually a submodel of its vectorial counterpart and, thus, any PCs found under the matrix model are also possible to discover under the vector model. Hence, the fact that the matrix model actually performed better than its more general vector version leads us to conclude that the ``regularization'' offered by the former was indeed beneficial in practice. Finally, we also applied HOSVD to the data, but the obtained leading PCs were heavily dominated by outliers and, thus, its results are not shown here. 

\section{Modification for sparse data}\label{sec:zero_inflated}

It is common for observed count data to contain more zero observations than the standard distributional models, such as the Poisson distribution, predict. This phenomenon is known as zero-inflation (or sparsity) and occurs, in particular, when the data is a mixture from two generating processes, of which one outputs only zero observations. For example, in abundance data, it could be that a certain species is biologically incompatible with a certain environment, leading automatically to a zero count. Motivated by this, in this section we present an ``experimental'' zero-inflated version of the model \eqref{eq:matrix_poisson_model_conditional} where the Poisson distribution is replaced with the zero-inflated Poisson distribution. Namely, we assume that the elements $x_{jk}$ of the observed $p_1 \times p_2$ random matrix $X$ are generated as $x_{jk} = b_{jk} y_{jk}$ where $b_{jk} \sim \mathrm{Ber}(\pi_{jk})$, $\pi_{jk} \in (0, 1]$, the random variables $y_{jk}$, $j=1,\ldots,p_1$, $k=1,\ldots,p_2$, are conditionally independent given $Z$ and satisfy,
\begin{align}\label{eq:matrix_poisson_model_conditional_sparse}
    y_{jk} \mid Z \sim \mathrm{Po}[ \exp \{ \mu_{jk} + (U_1 Z U_2')_{jk} \} ],
\end{align}
and $Z$, $\mu$, $U_1$ and $U_2$ are as in the model \eqref{eq:matrix_poisson_model_conditional}. Moreover, we assume that the mixture proportion variables $b_{jk}$, $j = 1, \ldots, p_1$, $k = 1, \ldots, p_2$, are independent of each other, the $y_{jk}$ and $Z$. Thus, the zero-inflated model \eqref{eq:matrix_poisson_model_conditional_sparse} is otherwise equivalent to its regular variant \eqref{eq:matrix_poisson_model_conditional}, but with the exception that each generated count is observed with the probability $\pi_{jk}$, and replaced with a zero otherwise.

The introduction of additional parameters now implies that the first and second-order quantities we used for the method of moments in Section \ref{subsec:parameter_estimation} must be supplemented with additional statistics for us to be able to estimate the parameters of the zero-inflated model. A natural choice would be to use the observed counts of zeroes for each cell but it turns out that their population versions lead to intractable integrals. Hence, we instead resort to third-order quantities. Indeed, using the techniques of the proof of Lemma \ref{lem:covariance_estimators}, it is straightforward to show that,
\begin{align}\label{eq:pi_estimator}
    \pi_{jk} = \frac{  \{ \mathrm{E} ( x_{j k} ) \}^3 \mathrm{E} \{ x_{j k} (x_{j k} - 1) (x_{j k} - 2) \} }{ [\mathrm{E} \{ x_{j k} (x_{j k} - 1) \}]^3},
\end{align}
allowing the estimation of $\pi_{jk}$. Let then $s_{1, jk}$, $j \neq k$, be defined as in \eqref{eq:left_side_s_elements} and let the diagonal elements $s_{1, jj}$ be defined as
\begin{align*}
    s_{1, jj} := \frac{1}{p_2} \sum_{\ell = 1}^{p_2} \log \left[ \pi_{j \ell} \frac{\mathrm{E} \{ x_{j \ell} (x_{j \ell} - 1) \} }{ \{ \mathrm{E}(x_{j \ell}) \}^2} \right].
\end{align*}
Defining also the elements $s_{2, jk}$ of the matrix $S_2$ analogously, we see that $ S_1 = \tau^2 U_1 \Lambda_1 U_1'$ and $S_2 = \tau^2 U_2 \Lambda_2 U_2'$, again allowing the estimation of $\tau^2$ through the identity $\tau^2 = \mathrm{tr}(S_1)/(2 p_1) + \mathrm{tr}(S_2)/(2 p_2)$. Finally, arguing again as in the proof of Lemma \ref{lem:covariance_estimators}, the elements $\mu_{jk}$ of the mean matrix $\mu$ are seen to satisfy 
\begin{align*}
    \mu_{j k} = -\frac{5}{2} \log \mathrm{E}(x_{jk}) + 4 \log \mathrm{E} \{ x_{j k} (x_{j k} - 1) \} - \frac{3}2{} \log \mathrm{E} \{ x_{j k} (x_{j k} - 1) (x_{j k} - 2) \}.
\end{align*}
Let next $\Pi$ be the $p_1 \times p_2$ matrix of the elements $\pi_{jk}$ and denote the sample estimators of the model parameters by $S_{n1}$, $S_{n2}$, $t_{n}^2$, $M_n$ and $P_n$, where the first four are as in Section \ref{subsec:parameter_estimation} (with the modifications mentioned above) and $P_n$ is the sample estimator of the matrix $\Pi$, containing the sample estimates of \eqref{eq:pi_estimator}. The next theorem says that the estimators are, similar to their counterparts in Section \ref{subsec:parameter_estimation}, root-$n$ consistent. Its proof is exactly analogous to that of Theorem~\ref{theo:asymptotics}, causing us to omit it. 

\begin{theorem}\label{theo:asymptotics_2}
    Under the model \eqref{eq:matrix_poisson_model_conditional_sparse}, each of the deviations,
    \begin{align*}
        S_{n1} - S_1, \quad S_{n2} - S_2, \quad t_n^2 - \tau^2 \quad M_n - \mu \quad \mbox{and} \quad P_n - \Pi,
    \end{align*}
    is of the order $\mathcal{O}_p(1/\sqrt{n})$ as $n \rightarrow \infty$.
\end{theorem}

While Theorem \ref{theo:asymptotics_2} guarantees that the zero-inflated estimators enjoy the convergence rate of their regular counterparts, their efficiency still leaves something to be desired. This results from the use of the third moments in the parameter estimates of the zero-inflated model, the higher moments introducing additional variance to the estimators. Besides the inflated variance, another practical issue with the estimator $P_n$ is that there is no guarantee that the estimators $p_{njk}$ of the $\pi_{jk}$ should lie in the interval $(0, 1]$ for any finite $n$, and in practice we found it convenient to threshold the estimates to the interval $[0.05, 1]$ which, of course, requires the assumption that the true probabilities satisfy $\pi_{jk} \geq 0.05$.

To further study this loss in efficiency, we conducted the following simulation experiment: The $4 \times 3$ observations $X_1, \ldots , X_n$ were generated from the zero-inflated model \eqref{eq:matrix_poisson_model_conditional_sparse} either with $S_1 = \tau^2 U_1 \Lambda_1 U_1' = I_4$, $S_2 = \tau^2 U_2 \Lambda_2 U_2' = I_3$ (the full rank model) or with $U_1 \Lambda_1 U_1' = 1_4 1_4'$, $U_2 \Lambda_2 U_2' = 1_3 1_3'$ (the low-rank model) where $1_p$ denotes the $p$-dimensional vector consisting solely of ones. In both cases we took $\mu = 0$. We considered several levels of zero-inflation, $\Pi = \pi 1_4 1_3'$, for $\pi = 1, 0.75, 0.5, 0.25$. Finally, we took two different sample sizes, $n = 500, 1000$, and each combination of the previous simulation settings was independently replicated 1000 times. In each of the replicates both the regular estimates from Section \ref{subsec:parameter_estimation} and the zero-inflated estimates of the parameters were computed, and for each estimate we measured its distance from the true parameter value in the Euclidean norm. No thresholding was used in the zero-inflated estimation of $P_n$. The average distances are shown for the full rank model in Table \ref{tab:simu_sparse_1} and for the low-rank model in Table \ref{tab:simu_sparse_2}.  In both tables, the column \textit{Est} describes which estimators were used, either the regular ones from Section \ref{subsec:parameter_estimation} (R) or the zero-inflated ones of this section (Z).

\begin{table}[t]
\centering
\caption{Results of the simulation study for the full rank model. The numbers give the average Euclidean distances of the parameter estimates from their true values. R refers to the regular estimators from Section \ref{subsec:parameter_estimation} and Z to the zero-inflated estimators.}
\label{tab:simu_sparse_1}
\begin{tabular}{cc|cccc|cccc|}
& & \multicolumn{4}{|c|}{$n = 500$} & \multicolumn{4}{|c|}{$n = 1000$} \\
  \hline
 $\pi$ & Est & $\mu$ & $S_1$ & $S_2$ & $\pi 1_4 1_3'$ & $\mu$ & $S_1$ & $S_2$ & $\pi 1_4 1_3'$ \\ 
  \hline
 \multirow{2}{*}{1.00} & R & 0.32 & 0.29 & 0.20 & - & 0.24 & 0.22 & 0.15 & -  \\ 
  & Z & 1.48 & 0.63 & 0.51 & 0.78 & 1.26 & 0.50 & 0.41 & 0.65 \\ 
 \hline
 \multirow{2}{*}{0.75} & R & 1.48 & 0.60 & 0.48 & - & 1.48 & 0.59 & 0.49 & -  \\ 
  & Z & 1.59 & 0.71 & 0.58 & 0.63 & 1.34 & 0.56 & 0.45 & 0.54 \\
  \hline
 \multirow{2}{*}{0.50} & R & 3.55 & 1.37 & 1.15 & - & 3.57 & 1.37 & 1.16 & -  \\ 
  & Z & 1.76 & 0.89 & 0.70 & 0.46 & 1.47 & 0.69 & 0.54 & 0.39 \\
  \hline
 \multirow{2}{*}{0.25} & R & 7.10 & 2.82 & 2.34 & - & 7.13 & 2.76 & 2.35 & -  \\ 
  & Z & 2.15 & 1.44 & 1.03 & 0.27 & 1.79 & 1.06 & 0.79 & 0.23 \\ 
   \hline
\end{tabular}
\end{table}

The results are very similar for both models, so we focus only on Table \ref{tab:simu_sparse_1} in the following. The table shows that decreasing the value of $\pi$ from 1.00 to 0.25 has the expected effects on the two sets of estimators: the regular estimators (R) are consistent only when no zero-inflation is present, i.e., $\pi = 1.00$, and they quickly deteriorate as $\pi$ decreases (visible in the fact that increasing the sample size $n$ does no good for them in cases other than $\pi = 1.00$). Whereas, while the estimation error does increase also for the zero-inflated estimator (Z) when $\pi$ is decreased, it is also the case that doubling $n$ from 500 to 1000 notably improves its performance, as expected. Additionally, we see that the matrix $\pi 1_4 1_3'$ is actually easier to estimate when $\pi$ is small which is probably a consequence of the finite sample performance of its estimator \eqref{eq:pi_estimator} (our experiments revealed that it has a strong downward bias). 

\begin{table}[t]
\centering
\caption{Results of the simulation study for the low-rank model. The numbers give the average Euclidean distances of the parameter estimates from their true values. R refers to the regular estimators from Section \ref{subsec:parameter_estimation} and Z to the zero-inflated estimators.}
\label{tab:simu_sparse_2}
\begin{tabular}{cc|cccc|cccc|}
& & \multicolumn{4}{|c|}{$n = 500$} & \multicolumn{4}{|c|}{$n = 1000$} \\
  \hline
 $\pi$ & Est & $\mu$ & $S_1$ & $S_2$ & $\pi 1_4 1_3'$ & $\mu$ & $S_1$ & $S_2$ & $\pi 1_4 1_3'$ \\ 
  \hline
 \multirow{2}{*}{1.00} & R & 0.29 & 0.53 & 0.39 & - & 0.21 & 0.41 & 0.31 & -  \\ 
  & Z & 1.42 & 0.81 & 0.65 & 0.75 & 1.21 & 0.66 & 0.53 & 0.62 \\
 \hline
 \multirow{2}{*}{0.75} & R & 1.47 & 0.82 & 0.63 & - & 1.47 & 0.73 & 0.57 & -  \\ 
  & Z & 1.52 & 0.92 & 0.72 & 0.61 & 1.29 & 0.74 & 0.58 & 0.53 \\
 \hline
 \multirow{2}{*}{0.50} & R & 3.54 & 1.54 & 1.24 & - & 3.57 & 1.48 & 1.22 & -  \\ 
  & Z & 1.74 & 1.11 & 0.83 & 0.46 & 1.45 & 0.88 & 0.66 & 0.39 \\
 \hline
 \multirow{2}{*}{0.25} & R & 7.11 & 3.09 & 2.49 & - & 7.14 & 2.93 & 2.43 & -  \\ 
  & Z & 2.11 & 1.87 & 1.30 & 0.27 & 1.75 & 1.39 & 0.96 & 0.23 \\ 
   \hline
\end{tabular}
\end{table}

Finally, we still discuss the estimation of the latent components under the zero-inflated model. The optimization strategy we used in Section \ref{subsec:variable_estimation} is no longer viable as the presence of the zero-inflation renders the logarithmic conditional density $\ell(z \mid x)$ non-concave. Namely, arguing as in the proof of Lemma~\ref{theo:logarithmic_conditional_density}, it is straightforwardly shown that $\ell(z | x)$ has the form,
\begin{align}\label{eq:conditional_density_zero_inflation}
	C + \sum_{j = 1}^p ( 1 - e_{jj} ) \log \{ \pi_j \exp(-h_j) + 1 - \pi_j \} + x' E U z - 1' E h  - \frac{1}{2 \tau^2} z' \Lambda^{-1} z,
\end{align}
where we have used the notation of Section \ref{subsec:variable_estimation}, $h := \exp ( m + U z )$, $\pi = (\pi_j)$ is taken to be the vectorization of the corresponding matrix $\Pi$, the constant $C$ does not depend on $z$ and $E$ is a $p \times p$ diagonal matrix whose $j$th diagonal element $e_{jj}$ is one if the corresponding element of the observation $x$ is non-zero, and zero otherwise. The second term in the expression \eqref{eq:conditional_density_zero_inflation} is what causes the non-concavity and, consequently, we instead use the Nelder-Mead method as implemented in the function \texttt{optimx} in the R-package \texttt{optimx} \citep{nash2011unifying} to maximize $\ell(z | x)$ for each of the observed $x_1, \ldots , x_n$. The obtained PCs $z_1, \ldots , z_n$ again need to be centered, as in Section \ref{subsec:variable_estimation}.

We next tested the estimation procedure by applying it to the matrix-valued abundance data we used in Section \ref{sec:examples}. Indeed, 52.6\% of all entries in the data are zero, implying that the assumption of zero-inflated data could be warranted. Applying the augmentation estimator to the data reveals (not shown here) that the latent column dimension is one and that the latent row dimension is one, or possibly two and, hence, we continue by extracting the principal components $z_{i, 11}$ and $z_{i, 21}$ from the data. The scatter plot of the components, with colors depicting the clusters described in Section \ref{sec:examples}, is shown in Figure \ref{fig:abundance_7} and clearly fails to sufficiently capture the group information in the data (only the \textit{Southern} cluster can be said to be somewhat separated from the rest). We conjecture that the reason for this is simply that the low sample size of the data is not sufficient to accommodate the zero-inflated model. Indeed, the large quantity of zero observations implies that only a certain fraction of the $n = 65$ observations came from the Poisson part of the model, leaving its parameter estimates inaccurate. As such, alternative ways should be pursued to incorporate the sparseness of the data in low-sample scenarios such as this, see also the discussion in Section \ref{sec:discussion}. This is, however, beyond the scope of the current work.

\begin{figure}
    \centering
    \includegraphics[width = 1\textwidth]{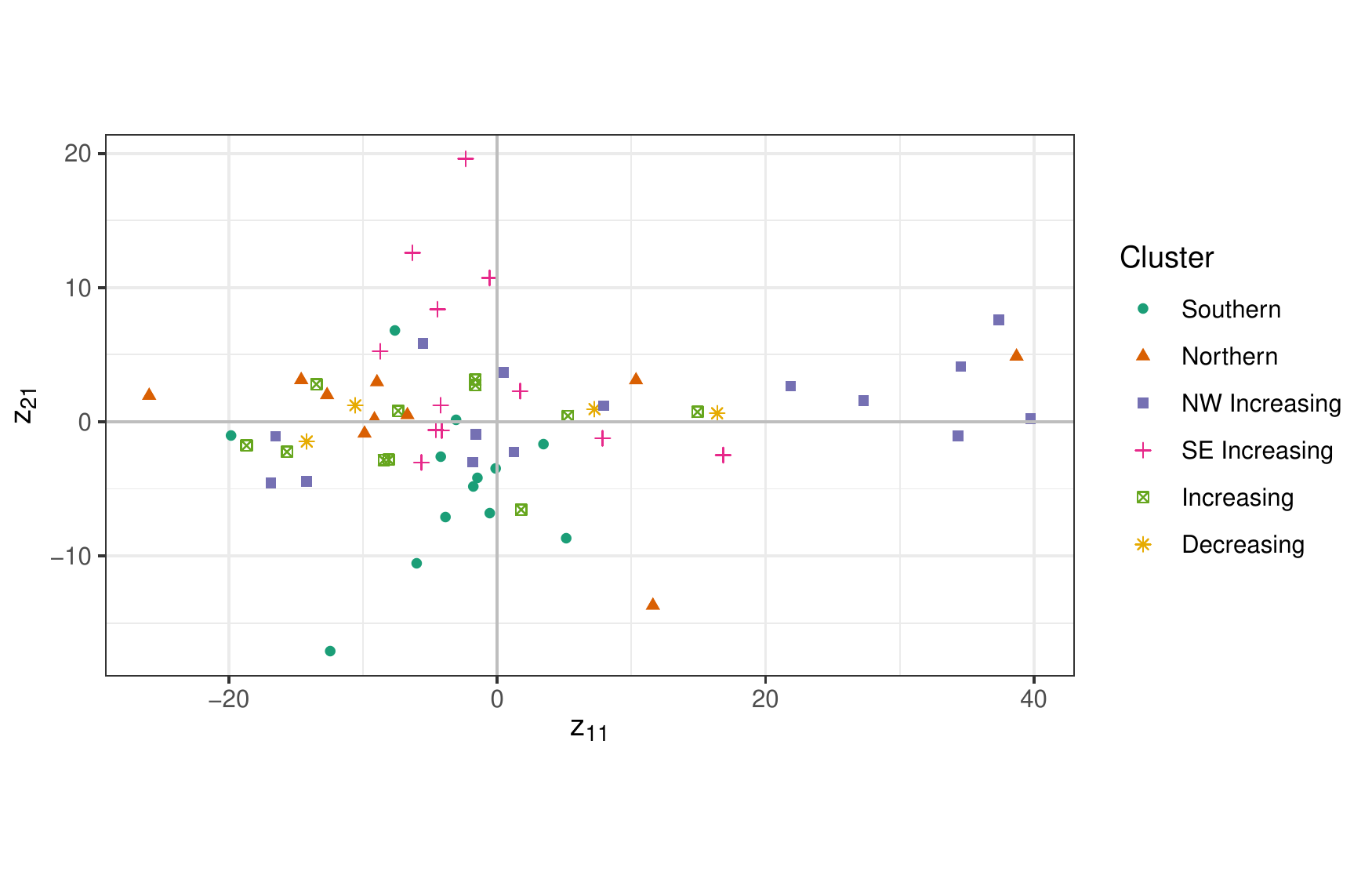}
    \caption{The scatter plot of the principal components $z_{i, 11}$ and $z_{i, 21}$ extracted from the matrix-valued abundance data under the assumption of the zero-inflated model.}
    \label{fig:abundance_7}
\end{figure}

\section{Discussion}\label{sec:discussion}

In this work, we proposed a latent variable model for data where the observations are matrices of counts. In addition to developing estimators for the model parameters and the principal components, we also proposed an efficient procedure for estimating the latent dimensions of the data.

The zero-inflated extension of the model proposed in Section \ref{sec:zero_inflated} was, while theoretically sound, observed to require rather large sample sizes to be useful in practice and thus more research is needed to obtain feasible alternatives that can accommodate also data with smaller $n$. Practicality aside, note that this requirement for larger sample sizes is rather natural in that the model~\eqref{eq:matrix_poisson_model_conditional_sparse} for sparse data has more parameters to estimate than the regular model~\eqref{eq:matrix_poisson_model_conditional}. The situation should also not be confused with that of a sparse \textit{model} (such as assumed in LASSO) where one has actually less effective parameters to estimate (due to the structural constraints imposed by the sparsity) and lower sample sizes suffice.

A natural continuation to the current work would be to extend the results to apply also to count-valued tensors (higher-order counterparts of matrices). Indeed, we expect that this could be rather straightforwardly carried out through the concept of tensor flattening, see  \cite{virta2017independent}, for a detailed derivation of a particular dimension reduction procedure, first for matrices and then to general tensors through the use of flattening. However, in the current work, we decided to limit ourselves to matrix data as: (i) examples of higher-order tensorial count data are still rather rare and, more importantly, (ii) the presentation of the theory is considerably less notationally intensive in the matrix case (cf. the two approaches in \cite{virta2017independent}).

Finally, as observed in Section \ref{sec:examples}, the Gaussian version of the augmentation procedure turned out to perform remarkably well in dimension estimation under the Poisson model, even though it quite severely violates the model assumptions. This interesting fact thus also warrants more study.





\appendix

\section{Limiting distribution of $S_1$ for $p_1 = p_2 = 1$}\label{sec:limiting_distribution}

For convenience, we introduce the following notation more suited to the current one-dimensional case: Let $z \sim \mathcal{N}(0, 1)$ and let $x \mid z \sim \mathrm{Po}\{ \exp(\mu + \sigma z) \}$. Moreover, with $x_1, \ldots , x_n$ denoting a random sample from the previous model, we let $m_{n1} := (1/n) \sum_{i=1}^n x_i $, $m_{n2} := (1/n) \sum_{i=1}^n x_i(x_i - 1)$, $m_1 := E(x)$ and $m_2 := E\{ x(x - 1) \}$. Consequently, our objective is to find the limiting distribution of
\begin{align*}
    h_n := \sqrt{n} \left\{ \log \left( \frac{m_{n2}}{m_{n1}^2} \right) - \log \left( \frac{m_{2}}{m_{1}^2} \right) \right\}.
\end{align*}
By the CLT, the limiting distribution of $\sqrt{n}(m_{n1} - m_1, m_{n2} - m_2)'$ is $\mathcal{N}_2(0, \Sigma)$, where
\begin{align*}
    \Sigma := \begin{pmatrix}
    \mathrm{Var}(x) & \mathrm{Cov}\{x, x(x - 1) \} \\
    \mathrm{Cov}\{x, x(x - 1) \} & \mathrm{Var}\{ x(x - 1) \}.
    \end{pmatrix}
\end{align*}
Arguing as in the proof of Lemma \ref{lem:covariance_estimators} and using the fact that the $j$th factorial moment of the Poisson distribution is equal to the $j$th power of its mean, we find that the $j$th factorial moment $m_j := \mathrm{E}\{ x ( x - 1 ) \cdots (x - j + 1) \}$ of $x$ satisfies $m_j = \exp \{ j \mu + (1/2) j^2 \sigma^2 \}$. The first four factorial and regular moments $t_j := \mathrm{E}(x^j)$ have the relationships $t_1 = m_1$, $t_2 = m_1 + m_2$, $t_3 = m_1 + 3 m_2 + m_3$ and $t_4 = m_1 + 7 m_2 + 6 m_3 + m_4$, implying that
\begin{align*}
    \Sigma = \begin{pmatrix}
    m_1 + m_2 - m_1^2 & 2 m_2 + m_3 - m_1 m_2 \\
    2 m_2 + m_3 - m_1 m_2 & 2 m_2 + 4 m_3 + m_4 - m_2^2.
    \end{pmatrix}
\end{align*}
As the gradient of the map $(a_1, a_2) \mapsto \log(a_2/a_1^2)$ is $(-2/a_1, 1/a_2)$, the delta method entails that the limiting distribution of $h_n$ is normal with zero mean and the variance
\begin{align*}
    & (-2 m_1^{-1}, m_2^{-1}) \Sigma (-2 m_1^{-1}, m_2^{-1})'\\
    =&  -4 m_1^{-1} + 4 m_2 m_1^{-2} - 4 m_3 m_1^{-1} m_2^{-1} + 2 m_2^{-1} + 4 m_3 m_2^{-2} + m_4 m_2^{-2} - 1.
\end{align*}
Plugging in the values $m_j = \exp \{ j \mu + (1/2) j^2 \sigma^2 \}$ now reveals that the asymptotic variance does not indeed simplify any further.

\section{Proofs}\label{sec:proofs}

\begin{proof}[Proof of Lemma \ref{lem:covariance_estimators}]
    It is sufficient to show the claim for $S_1$, as the result for $S_2$ follows instantly after transposition of the model.
    
    By the law of total expectation,
    \begin{align*}
        \mathrm{E}(x_{j\ell}) = \mathrm{E} \{ \mathrm{E}(x_{j\ell} \mid Z) \} = \mathrm{E} \{ \exp( \mu_{j\ell} + u_{1,j}' Z u_{2, \ell}) \},
    \end{align*}
    where $u_{a,j} \in \mathbb{R}^{d_a}$ denotes the $j$th row of $U_a$, for $a = 1, 2$. The distribution of $u_{1,j}' Z u_{2, \ell}$ is $\mathcal{N}(0, \tau^2 u_{1,j}' \Lambda_1 u_{1,j} u_{2,\ell}' \Lambda_2 u_{2,\ell})$, showing that
    \begin{align*}
        \mathrm{E}(x_{j\ell}) = \exp( \mu_{j\ell} + (1/2) \tau^2 u_{1,j}' \Lambda_1 u_{1,j} u_{2,\ell}' \Lambda_2 u_{2,\ell} ).
    \end{align*}
    One can similarly establish that
    \begin{align*}
        \mathrm{E} \{ x_{j \ell} (x_{j \ell} - 1) \} = \exp( 2 \mu_{j\ell} + 2 \tau^2 u_{1,j}' \Lambda_1 u_{1,j} u_{2,\ell}' \Lambda_2 u_{2,\ell} ),
    \end{align*}
    and that
    \begin{align*}
        \mathrm{E} ( x_{j \ell} x_{k \ell} ) = \exp \{ \mu_{j\ell} + \mu_{k\ell} + (1/2) \tau^2 (u_{1,j} + u_{1,k})' \Lambda_1 (u_{1,j} + u_{1,k}) u_{2,\ell}' \Lambda_2 u_{2,\ell} \}, 
    \end{align*}
    where the latter result assumes $j \neq k$ and uses the conditional independence of $ x_{j \ell} $ and $ x_{k \ell} $ given $Z$. Plugging now in to the definitions of $s_{1, jk}$ and $s_{1, jj}$ in \eqref{eq:left_side_s_elements}  gives
    \begin{align*}
        s_{1, jk} = \frac{1}{p_2} \sum_{\ell= 1}^{p_2} \tau^2 u_{1,j}' \Lambda_1 u_{1,k} u_{2,\ell}' \Lambda_2 u_{2,\ell} = \frac{1}{p_2} \mathrm{tr} ( U_2 \Lambda_2 U_2') \tau^2 u_{1,j}' \Lambda_1 u_{1,k} = \tau^2 u_{1,j}' \Lambda_1 u_{1,k},
    \end{align*}
    showing the claim for the off-diagonal elements $s_{1, jk}$. A similar plug-in the gives the analogous result also for the diagonal elements $s_{1, jj}$.  
\end{proof}

\begin{proof}[Proof of Theorem \ref{theo:asymptotics}]
    We show the result only for $S_{n1}$, the rest following analogously (or, in the case of $t_n^2$, from the convergence of $S_{n1}$ and $S_{n2}$). As the $k$th factorial moment of a $\mathrm{Po}(\lambda)$-variate equals $\lambda^k$, computations analogous to the ones in the proof of Lemma \ref{lem:covariance_estimators} show that factorial moments and cross-moments of all orders exist as finite for the elements of $X$. As a consequence of this, also raw moments and cross-moments of all orders exist and are finite for the elements of $X$. Consequently, by the central limit theorem, the vector $(1/n) \sum_{i = 1}^n ( \ldots, x_{i, jk}, \ldots , x_{i, jk}\{x_{i, jk} - 1\}, \ldots x_{i, j k} x_{i, \ell t}, \ldots )'$ of all sample factorial moments of the first and second order is asymptotically normal.
    
    Consider now the maps $g_1: \mathbb{R}^3 \setminus \{ 0 \} \to \mathbb{R}$ and $g_2: \mathbb{R}^2 \setminus \{ 0 \} \to \mathbb{R}$ acting as $(y_1, y_2, y_3)' \mapsto \log \{ y_1/(y_2 y_3) \} $ and $(y_1, y_2)' \mapsto \log ( y_1/y_2^2 ) $. The gradients of the maps are $\nabla g_1(y_1, y_2, y_3) =  (1/y_1, -1/y_2, -1/y_3)$ and $\nabla g_2(y_1, y_2) =  (1/y_1, -2/y_2)$. We also have, by the proof of Lemma~\ref{lem:covariance_estimators}, that all expected values of the forms $\mathrm{E}(x_{j\ell})$, $ \mathrm{E}(x_{j\ell} x_{k \ell}) $ and $ \mathrm{E}\{ x_{j\ell} (x_{j\ell} - 1) \} $ are positive under the model~\eqref{eq:matrix_poisson_model_conditional}. Consequently, by the delta method, all quantities of the forms
    \begin{align*}
        \log \left\{ \frac{\frac{1}{n} \sum_{i=1}^n x_{i, j\ell} x_{i, k \ell}}{ \left( \frac{1}{n} \sum_{i=1}^n x_{i, j\ell} \right) \left( \frac{1}{n} \sum_{i=1}^n x_{i, k \ell} \right)} \right\} \quad \mbox{and} \quad \log \left\{  \frac{\frac{1}{n} \sum_{i=1}^n x_{i, j\ell} (x_{i, j\ell} - 1)}{ \left( \frac{1}{n} \sum_{i=1}^n x_{i, j\ell} \right)^2} \right\}
    \end{align*}
    have a joint limiting normal distribution (with the root-$n$ convergence rate), from which it follows that $S_{n1}$ has a limiting normal distribution, thus proving the claim.
\end{proof}

\begin{proof}[Proof of Lemma \ref{lem:normal_conditional_distribution}]
    The vectorization of the model \eqref{eq:matrix_normal_model} reads,
    \begin{align*}
        x = m + U z + \mathrm{vec}(\varepsilon),
    \end{align*}
    where $ z \sim \mathcal{N}_d (0, \tau^2 \Lambda) $ is independent of $\mathrm{vec}(\varepsilon) \sim \mathcal{N}_p (0, \sigma^2 I_p)$. Consequently, the joint distribution of $ (z', x')' $ is
	\[ 
	\begin{pmatrix}
	z \\
	x
	\end{pmatrix} \sim
	\mathcal{N}_{d + p} \left\{ \begin{pmatrix}
	0 \\
	m
	\end{pmatrix}, \begin{pmatrix}
	\tau^2 \Lambda & \tau^2\Lambda U' \\
	\tau^2 U \Lambda & \tau^2 U \Lambda U' + \sigma^2 I_p
	\end{pmatrix} \right\}.
	\]
	Now, $ ( \tau^2 \Lambda U' ) ( \tau^2 U \Lambda U' + \sigma^2 I_p )^{-1} = \tau^2 \Lambda (\tau^2 \Lambda + \sigma^2 I_d)^{-1} U' $, and by the standard properties of Gaussian conditional distributions,
	\[ 
	z \mid x \sim \mathcal{N}_d \{ a(x), B \},
	\]
	where $a(x) := \tau^2 \Lambda ( \tau^2 \Lambda + \sigma^2 I_d)^{-1} U' ( x - m )$ and $B := \{ I_d - \tau^2 \Lambda ( \tau^2 \Lambda + \sigma^2 I_d)^{-1} \} \tau^2 \Lambda$. The result now follows.
\end{proof}

\begin{proof}[Proof of Theorem \ref{theo:logarithmic_conditional_density}]

    The Bayes rule implies that,
    \begin{align}\label{eq:bayes_rule}
        f_{z|x}(z|x) = C_0 f_{x|z}( x \mid z ) f_z ( z ) ,   
    \end{align}
    for some positive constant $C_0$ not depending on $z$. Then, indexing the elements of $x$ as $x_j$, $j = 1, \ldots , p$, we have,
    \begin{align*}
    \log f_{x|z} (x \mid z) = - \sum_{j = 1}^p \log x_j! + x' h(z) - 1'\exp\{ h(z) \},
    \end{align*}
    where the exponential function is taken element-wise and we denote $h(z) := m + U z$. Additionally,
    \begin{align*}
        \log f_z ( z ) = - \frac{d}{2} \log 2 \pi - \frac{1}{2} \tau^{2d} \log |\Lambda| - \frac{1}{2 \tau^2} z' \Lambda^{-1} z.
    \end{align*}
    Plugging the previous two formulas into \eqref{eq:bayes_rule} now yields result \textit{i)}.
    
    To see \textit{ii)}, we fix $x$ and first establish the strict concavity by showing that the Hessian of $z \mapsto \ell (z | x)$ is negative definite. The gradient of the map is
    \begin{align*}
        \nabla \ell(z | x) = U' x - U' \exp \{ h(z) \} - \tau^{-2} \Lambda^{-1} z,
    \end{align*}
    giving further the Hessian,
    \begin{align*}
        \nabla^2 \ell(z | x) =  - U' \mathrm{diag}[\exp \{ h(z) \}] U - \tau^{-2} \Lambda^{-1}.
    \end{align*}
    The diagonal matrices $\mathrm{diag}[\exp \{ h(z) \}]$ and $\Lambda^{-1}$ have strictly positive diagonal elements (and $\tau > 0$), making the Hessian negative definite and implying that $\ell (z | x)$ is indeed strictly concave in $z$, for all $x$.
    
    We next show that $z \mapsto \ell (z | x)$ is coercive in the sense that, for all sequences $z_n \in \mathbb{R}^d$ such that $\| z_n \| \rightarrow \infty$, we have $\ell(z_n | x) \rightarrow -\infty$. Combined with the continuity of $ z \mapsto \ell (z | x) $, the coercivity will then imply that the function admits at least one global maximizer, and strict concavity then guarantees that the maximizer is unique.
    
    Let thus $z_n \in \mathbb{R}^d$ be such that $\| z_n \| \rightarrow \infty$ and write $\alpha_n := \| z_n \| \rightarrow \infty$ and $u_n := z_n/\| z_n \|$. Then,
    \begin{align*}
        \ell(z_n | x) &= C + \alpha_n x' U u_n - 1' \exp \{ h( \alpha_n u_n ) \} - \frac{\alpha_n^2}{2 \tau^2} u_n' \Lambda^{-1} u_n \\
        &\leq C + \alpha_n x' U u_n  - \frac{\alpha_n^2}{2 \tau^2} u_n' \Lambda^{-1} u_n\\
        &\leq C + \alpha_n \| U' x \| - \frac{\alpha_n^2}{2 \tau^2} \phi_{d}( \Lambda^{-1} ),
    \end{align*}
    where $\phi_{d}( \Lambda^{-1} ) > 0$ denotes the smallest eigenvalue of the positive definite matrix $ \Lambda^{-1} $. The derived upper bound is dominated 
    by the quadratic term, thus guaranteeing that $\ell(z_n | x) \rightarrow -\infty$, as desired.
\end{proof}

\bibliographystyle{chicago}

\bibliography{references_poisson}

\begin{thebibliography}{}

\bibitem[\protect\citeauthoryear{Aitchison and Ho}{Aitchison and
  Ho}{1989}]{aitchison1989multivariate}
Aitchison, J. and C.~Ho (1989).
\newblock The multivariate {P}oisson-log normal distribution.
\newblock {\em Biometrika\/}~{\em 76\/}(4), 643--653.

\bibitem[\protect\citeauthoryear{Bally and Caramellino}{Bally and
  Caramellino}{2016}]{bally2016asymptotic}
Bally, V. and L.~Caramellino (2016).
\newblock Asymptotic development for the {CLT} in total variation distance.
\newblock {\em Bernoulli\/}~{\em 22\/}(4), 2442--2485.

\bibitem[\protect\citeauthoryear{Chiquet, Mariadassou, and Robin}{Chiquet
  et~al.}{2018}]{chiquet2018variational}
Chiquet, J., M.~Mariadassou, and S.~Robin (2018).
\newblock Variational inference for probabilistic {P}oisson {PCA}.
\newblock {\em Annals of Applied Statistics\/}~{\em 12\/}(4), 2674--2698.

\bibitem[\protect\citeauthoryear{Collins, Dasgupta, and Schapire}{Collins
  et~al.}{2001}]{collins2001generalization}
Collins, M., S.~Dasgupta, and R.~E. Schapire (2001).
\newblock A generalization of principal components analysis to the exponential
  family.
\newblock In {\em NIPS 2001}, Volume~13, pp.\ ~23.

\bibitem[\protect\citeauthoryear{De~Lathauwer, De~Moor, and
  Vandewalle}{De~Lathauwer et~al.}{2000}]{de2000multilinear}
De~Lathauwer, L., B.~De~Moor, and J.~Vandewalle (2000).
\newblock A multilinear singular value decomposition.
\newblock {\em SIAM Journal on Matrix Analysis and Applications\/}~{\em
  21\/}(4), 1253--1278.

\bibitem[\protect\citeauthoryear{Ding and Cook}{Ding and
  Cook}{2014}]{ding2014dimension}
Ding, S. and R.~D. Cook (2014).
\newblock Dimension folding {PCA} and {PFC} for matrix-valued predictors.
\newblock {\em Statistica Sinica\/}~{\em 24\/}(1), 463--492.

\bibitem[\protect\citeauthoryear{Ding and Cook}{Ding and
  Cook}{2015}]{ding2015tensor}
Ding, S. and R.~D. Cook (2015).
\newblock Tensor sliced inverse regression.
\newblock {\em Journal of Multivariate Analysis\/}~{\em 133}, 216--231.

\bibitem[\protect\citeauthoryear{Eaton and Tyler}{Eaton and
  Tyler}{1991}]{eaton1991wielandt}
Eaton, M.~L. and D.~E. Tyler (1991).
\newblock On {W}ielandt's inequality and its application to the asymptotic
  distribution of the eigenvalues of a random symmetric matrix.
\newblock {\em Annals of Statistics\/}~{\em 19\/}(1), 260--271.

\bibitem[\protect\citeauthoryear{Frelat, Lindegren, Denker, Floeter, Fock,
  Sguotti, St{\"a}bler, Otto, and M{\"o}llmann}{Frelat
  et~al.}{2017}]{frelat2017community}
Frelat, R., M.~Lindegren, T.~S. Denker, J.~Floeter, H.~O. Fock, C.~Sguotti,
  M.~St{\"a}bler, S.~A. Otto, and C.~M{\"o}llmann (2017).
\newblock Community ecology in {3D}: Tensor decomposition reveals
  spatio-temporal dynamics of large ecological communities.
\newblock {\em PloS one\/}~{\em 12\/}(11), e0188205.

\bibitem[\protect\citeauthoryear{Gupta and Nagar}{Gupta and
  Nagar}{2018}]{gupta2018matrix}
Gupta, A.~K. and D.~K. Nagar (2018).
\newblock {\em Matrix Variate Distributions}, Volume 104.
\newblock CRC Press.

\bibitem[\protect\citeauthoryear{Hall, Ormerod, and Wand}{Hall
  et~al.}{2011}]{hall2011theory}
Hall, P., J.~T. Ormerod, and M.~P. Wand (2011).
\newblock Theory of {G}aussian variational approximation for a {P}oisson mixed
  model.
\newblock {\em Statistica Sinica\/}~{\em 21\/}(1), 369--389.

\bibitem[\protect\citeauthoryear{Hu, Rai, Chen, Harding, and Carin}{Hu
  et~al.}{2015}]{hu2015scalable}
Hu, C., P.~Rai, C.~Chen, M.~Harding, and L.~Carin (2015).
\newblock Scalable {B}ayesian non-negative tensor factorization for massive
  count data.
\newblock In {\em Joint European Conference on Machine Learning and Knowledge
  Discovery in Databases}, pp.\  53--70. Springer.

\bibitem[\protect\citeauthoryear{Hung, Wu, Tu, and Huang}{Hung
  et~al.}{2012}]{hung2012multilinear}
Hung, H., P.~Wu, I.~Tu, and S.~Huang (2012).
\newblock On multilinear principal component analysis of order-two tensors.
\newblock {\em Biometrika\/}~{\em 99\/}(3), 569--583.

\bibitem[\protect\citeauthoryear{Izs{\'a}k}{Izs{\'a}k}{2008}]{izsak2008maximum}
Izs{\'a}k, R. (2008).
\newblock Maximum likelihood fitting of the {P}oisson lognormal distribution.
\newblock {\em Environmental and Ecological Statistics\/}~{\em 15\/}(2),
  143--156.

\bibitem[\protect\citeauthoryear{Kenney, Gu, and Huang}{Kenney
  et~al.}{2019}]{kenney2019poisson}
Kenney, T., H.~Gu, and T.~Huang (2019).
\newblock Poisson {PCA}: Poisson measurement error corrected {PCA}, with
  application to microbiome data.
\newblock {\em Biometrics\/}.

\bibitem[\protect\citeauthoryear{Kolda and Bader}{Kolda and
  Bader}{2009}]{kolda2009tensor}
Kolda, T.~G. and B.~W. Bader (2009).
\newblock Tensor decompositions and applications.
\newblock {\em SIAM Review\/}~{\em 51\/}(3), 455--500.

\bibitem[\protect\citeauthoryear{Landgraf}{Landgraf}{2015}]{landgraf2015generalized}
Landgraf, A.~J. (2015).
\newblock {\em Generalized principal component analysis: dimensionality
  reduction through the projection of natural parameters}.
\newblock Ph.\ D. thesis, Ohio State University.

\bibitem[\protect\citeauthoryear{Leibovici and Sabatier}{Leibovici and
  Sabatier}{1998}]{leibovici1998singular}
Leibovici, D. and R.~Sabatier (1998).
\newblock A singular value decomposition of a $k$-way array for a principal
  component analysis of multiway data, {PTA}-$k$.
\newblock {\em Linear Algebra and its Applications\/}~{\em 269\/}(1-3),
  307--329.

\bibitem[\protect\citeauthoryear{Li, Kim, and Altman}{Li
  et~al.}{2010}]{li2010dimension}
Li, B., M.~K. Kim, and N.~Altman (2010).
\newblock On dimension folding of matrix-or array-valued statistical objects.
\newblock {\em Annals of Statistics\/}~{\em 38\/}(2), 1094--1121.

\bibitem[\protect\citeauthoryear{Li and Tao}{Li and Tao}{2010}]{li2010simple}
Li, J. and D.~Tao (2010).
\newblock Simple exponential family {PCA}.
\newblock In {\em Proceedings of the Thirteenth International Conference on
  Artificial Intelligence and Statistics}, pp.\  453--460. JMLR Workshop and
  Conference Proceedings.

\bibitem[\protect\citeauthoryear{Liu, Chen, and Li}{Liu
  et~al.}{2019}]{liu2019time}
Liu, S., Z.~Chen, and X.~Li (2019).
\newblock Time-semantic-aware {P}oisson tensor factorization approach for
  scalable hotel recommendation.
\newblock {\em Information Sciences\/}~{\em 504}, 422--434.

\bibitem[\protect\citeauthoryear{Luo and Li}{Luo and Li}{2020}]{luo2020order}
Luo, W. and B.~Li (2020).
\newblock On order determination by predictor augmentation.
\newblock {\em Biometrika\/}~{\em 108\/}(3), 557--574.

\bibitem[\protect\citeauthoryear{Nash, Varadhan, et~al.}{Nash
  et~al.}{2011}]{nash2011unifying}
Nash, J.~C., R.~Varadhan, et~al. (2011).
\newblock Unifying optimization algorithms to aid software system users: optimx
  for {R}.
\newblock {\em Journal of Statistical Software\/}~{\em 43\/}(9), 1--14.

\bibitem[\protect\citeauthoryear{Niku, Brooks, Herliansyah, Hui, Taskinen,
  Warton, and {van der Veen}}{Niku et~al.}{2020}]{r_gllvm}
Niku, J., W.~Brooks, R.~Herliansyah, F.~K. Hui, S.~Taskinen, D.~I. Warton, and
  B.~{van der Veen} (2020).
\newblock {\em gllvm: Generalized Linear Latent Variable Models}.
\newblock R package version 1.2.3.

\bibitem[\protect\citeauthoryear{Niku, Warton, Hui, and Taskinen}{Niku
  et~al.}{2017}]{niku2017generalized}
Niku, J., D.~I. Warton, F.~K. Hui, and S.~Taskinen (2017).
\newblock Generalized linear latent variable models for multivariate count and
  biomass data in ecology.
\newblock {\em Journal of Agricultural, Biological and Environmental
  Statistics\/}~{\em 22\/}(4), 498--522.

\bibitem[\protect\citeauthoryear{Nordhausen and Tyler}{Nordhausen and
  Tyler}{2015}]{nordhausen2015cautionary}
Nordhausen, K. and D.~E. Tyler (2015).
\newblock A cautionary note on robust covariance plug-in methods.
\newblock {\em Biometrika\/}~{\em 102\/}(3), 573--588.

\bibitem[\protect\citeauthoryear{{R Core Team}}{{R Core Team}}{2020}]{Rbase}
{R Core Team} (2020).
\newblock {\em R: A Language and Environment for Statistical Computing}.
\newblock Vienna, Austria: R Foundation for Statistical Computing.

\bibitem[\protect\citeauthoryear{Radoji\v{c}i\'c, Lietz\'en, Nordhausen, and
  Virta}{Radoji\v{c}i\'c et~al.}{2021}]{radojicic2021estimating}
Radoji\v{c}i\'c, U., N.~Lietz\'en, K.~Nordhausen, and J.~Virta (2021).
\newblock On estimating the latent dimension in two-dimensional {PCA}.
\newblock In {\em Proceedings of the 12 International Symposium on Image and
  Signal Processing and Analysis (ISPA 2021)}, pp.\  16--22.

\bibitem[\protect\citeauthoryear{Schein, Paisley, Blei, and Wallach}{Schein
  et~al.}{2015}]{schein2015bayesian}
Schein, A., J.~Paisley, D.~M. Blei, and H.~Wallach (2015).
\newblock Bayesian {P}oisson tensor factorization for inferring multilateral
  relations from sparse dyadic event counts.
\newblock In {\em Proceedings of the 21th ACM SIGKDD International Conference
  on Knowledge Discovery and Data Mining}, pp.\  1045--1054.

\bibitem[\protect\citeauthoryear{Smallman, Artemiou, and Morgan}{Smallman
  et~al.}{2018}]{smallman2018sparse}
Smallman, L., A.~Artemiou, and J.~Morgan (2018).
\newblock Sparse generalised principal component analysis.
\newblock {\em Pattern Recognition\/}~{\em 83}, 443--455.

\bibitem[\protect\citeauthoryear{Smallman, Underwood, and Artemiou}{Smallman
  et~al.}{2020}]{smallman2019simple}
Smallman, L., W.~Underwood, and A.~Artemiou (2020).
\newblock Simple poisson {PCA}: an algorithm for (sparse) feature extraction
  with simultaneous dimension determination.
\newblock {\em Computational Statistics\/}~{\em 35}, 559--577.

\bibitem[\protect\citeauthoryear{Tyler}{Tyler}{1981}]{tyler1981asymptotic}
Tyler, D.~E. (1981).
\newblock Asymptotic inference for eigenvectors.
\newblock {\em Annals of Statistics\/}~{\em 9\/}(4), 725--736.

\bibitem[\protect\citeauthoryear{Tyler, Critchley, D{\"u}mbgen, and Oja}{Tyler
  et~al.}{2009}]{tyler2009invariant}
Tyler, D.~E., F.~Critchley, L.~D{\"u}mbgen, and H.~Oja (2009).
\newblock Invariant co-ordinate selection.
\newblock {\em Journal of the Royal Statistical Society: Series B (Statistical
  Methodology)\/}~{\em 71\/}(3), 549--592.

\bibitem[\protect\citeauthoryear{Virta, Koesner, Li, Nordhausen, and Oja}{Virta
  et~al.}{2016}]{rtensorBSS}
Virta, J., C.~L. Koesner, B.~Li, K.~Nordhausen, and H.~Oja (2016).
\newblock {\em tensorBSS: Blind Source Separation Methods for Tensor-Valued
  Observations}.
\newblock R package version 0.3.8.

\bibitem[\protect\citeauthoryear{Virta, Li, Nordhausen, and Oja}{Virta
  et~al.}{2017}]{virta2017independent}
Virta, J., B.~Li, K.~Nordhausen, and H.~Oja (2017).
\newblock Independent component analysis for tensor-valued data.
\newblock {\em Journal of Multivariate Analysis\/}~{\em 162}, 172--192.

\bibitem[\protect\citeauthoryear{Zhang and Zhou}{Zhang and
  Zhou}{2005}]{zhang20052d}
Zhang, D. and Z.-H. Zhou (2005).
\newblock {(2D)$^2$PCA}: Two-directional two-dimensional pca for efficient face
  representation and recognition.
\newblock {\em Neurocomputing\/}~{\em 69\/}(1-3), 224--231.

\end{thebibliography}

\end{document}